\newif\ifdraft
\draftfalse
 
\documentclass[3p, times]{elsarticle}

\usepackage[utf8]{inputenc}
\usepackage[english]{babel}
 
\usepackage[utf8]{inputenc}
\usepackage[T1]{fontenc}
\usepackage{array}
\usepackage{amssymb}
\usepackage{amsfonts}
\usepackage{amsmath}
\usepackage{amsthm}
\usepackage{graphicx}
\usepackage{subfig} 
\usepackage{textcomp}
\usepackage{bm}
\usepackage{booktabs}
\usepackage{mathtools}

\DeclarePairedDelimiter\floor{\lfloor}{\rfloor}

\newtheorem{theorem}{Theorem}

\newtheorem{assumptions}{Assumptions}
\newtheorem{lemma}[theorem]{Lemma}
\newtheorem{proposition}{Proposition}

\theoremstyle{definition}
\newtheorem{definition}{Definition}[section]

\newcommand{\R}{\mathbb{R}}

\renewcommand{\P}{\mathbb{P}}
\newcommand{\E}{\mathbb{E}}

\newcommand{\e}{\mathrm{e}}
\newcommand{\rmd}{\mathrm{d}}

\newcommand{\Cdop}{C_\mathrm{dop}}
\newcommand{\DSi}{D_\mathrm{Si}}
\newcommand{\Dox}{D_\mathrm{ox}}
\newcommand{\Dliq}{D_\mathrm{liq}}
\newcommand{\unit}[2]{#1\,\mathrm{#2}}

\DeclareMathOperator{\RMSE}{RMSE}

\DeclareMathOperator{\EMC}{E_{MC}}
\DeclareMathOperator{\EMLMC}{E_{MLMC}}
\usepackage{color}

\ifdraft
\usepackage{color}
 \usepackage[color]{showkeys}
\usepackage{showlabels}
\usepackage[draft,disable]{todonotes}

\else
\newcommand{\todo}[1]{}

\fi

\newcommand{\RN}[1]{%
  \textup{\uppercase\expandafter{\romannumeral#1}}%
}
 
\usepackage[figuresright]{rotating}

\begin{document}

\begin{frontmatter}


\title{The optimal multilevel Monte-Carlo approximation of the
  stochastic drift-diffusion-Poisson system }
\author[tuwien]{Leila Taghizadeh\corref{cor1}}
\ead{Leila.Taghizadeh@TUWien.ac.at}
\author[tuwien]{Amirreza Khodadadian}
\ead{Amirreza.Khodadadian@TUWien.ac.at}
\author[tuwien,asu]{Clemens Heitzinger}
\ead{Clemens.Heitzinger@TUWien.ac.at}
\cortext[cor1]{Corresponding author}
 
\address[tuwien]{Institute for Analysis and Scientific Computing, 
  Vienna University of Technology (TU Wien),
  Wiedner Hauptstraße 8--10, 
  1040 Vienna, Austria}
\address[asu]{School of Mathematical and Statistical Sciences,
  Arizona State University, Tempe, AZ 85287, USA}

\date{}

\begin{abstract}
 Existence and local-uniqueness theorems for weak solutions of a
   system consisting of the drift-diffusion-Poisson equations and the
   Poisson-Boltzmann equation, all with stochastic coefficients, are
   presented.  For the numerical approximation of the expected value of
   the solution of the system, we develop a multi-level Monte-Carlo
   (MLMC) finite-element method (FEM) and we analyze its rate of
   convergence and its computational complexity.  This allows to find
   the optimal choice of discretization parameters.  Finally, numerical
   results show the efficiency of the method.  Applications are, among
   others, noise and fluctuations in nanoscale transistors, in
   field-effect bio- and gas sensors, and in nanopores.
\end{abstract}


\end{frontmatter}

\section{Introduction}\label{s:intro}

In this work, we consider the system consisting of the
drift-diffusion-Poisson equations coupled with the Poisson-Boltzmann
equation, all with random coefficients.  We show existence and local
uniqueness of weak solutions for the stationary problem.  This system
is a general model for transport processes, where a stochastic process
determines the coefficients.  Furthermore, we develop a multi-level
(ML) Monte-Carlo (MC) finite-element method (FEM) for the system of
equations.  The different types of errors in the numerical
approximation must be balanced and the optimal approach is found here.

In the system of equations considered here, both the operators and the
forcing terms are stochastic, and therefore this system has numerous
applications (see Figure \ref{f:schematic}).  A deterministic and
simplified version, without the Poisson-Boltzmann equation, is the
standard model for semiconductor devices.  Nowadays, randomness due to
the location of impurity atoms is the most important effect limiting
the design of integrated circuits.  This application area is included
in the present model equations.  Furthermore, the full system of
equations considered here describes a very general class of
field-effect sensors including their most recent incarnation, nanowire
bio- and gas sensors.  While previous mathematical modeling has
focused on the deterministic problem and stochastic surface reactions
\cite{heitzinger2010multiscale, baumgartner2012existence,
  Baumgartner2013one-level, Baumgartner2013predictive,
  Heitzinger2014multiscale, Tulzer2015fluctuations}, the present model
describes how various stochastic processes propagate through a PDE
model and result in noise and fluctuations in a transport model.
Quantifying noise and fluctuations in sensors is important, since they
determine the detection limit and the signal-to-noise ratio.  Noise
and fluctuations are of great importance especially in nanometer-scale
devices, as any random effect becomes proportionally more important as
devices are shrunk.

Various sources of noise and fluctuations are included in the model
equations here.  Doping of semiconductor devices is inherently random
and results in a random number of impurity atoms placed at random
positions, each one changing the charge concentration and the mobility
at its location.  In field-effect sensors, target molecules bind to
randomly placed probe molecules in a stochastic process, so that the
detection mechanism is inherently stochastic.  The Brownian motion of
the target molecules also results in changes in charge concentration
and permittivity.  This randomness at the sensor surface propagates
through the self-consistent transport equations and finally results in
noise in the sensor output.

In summary, there are many applications where both the operators and
the forcing terms in the drift-diffusion-Poisson system are random.
The probability distributions of permittivities and charge
concentrations can be calculated from physical models
\cite{heitzinger2010calculation}.

\begin{figure}[t!]
  \centering
  \includegraphics[width=0.75\linewidth]{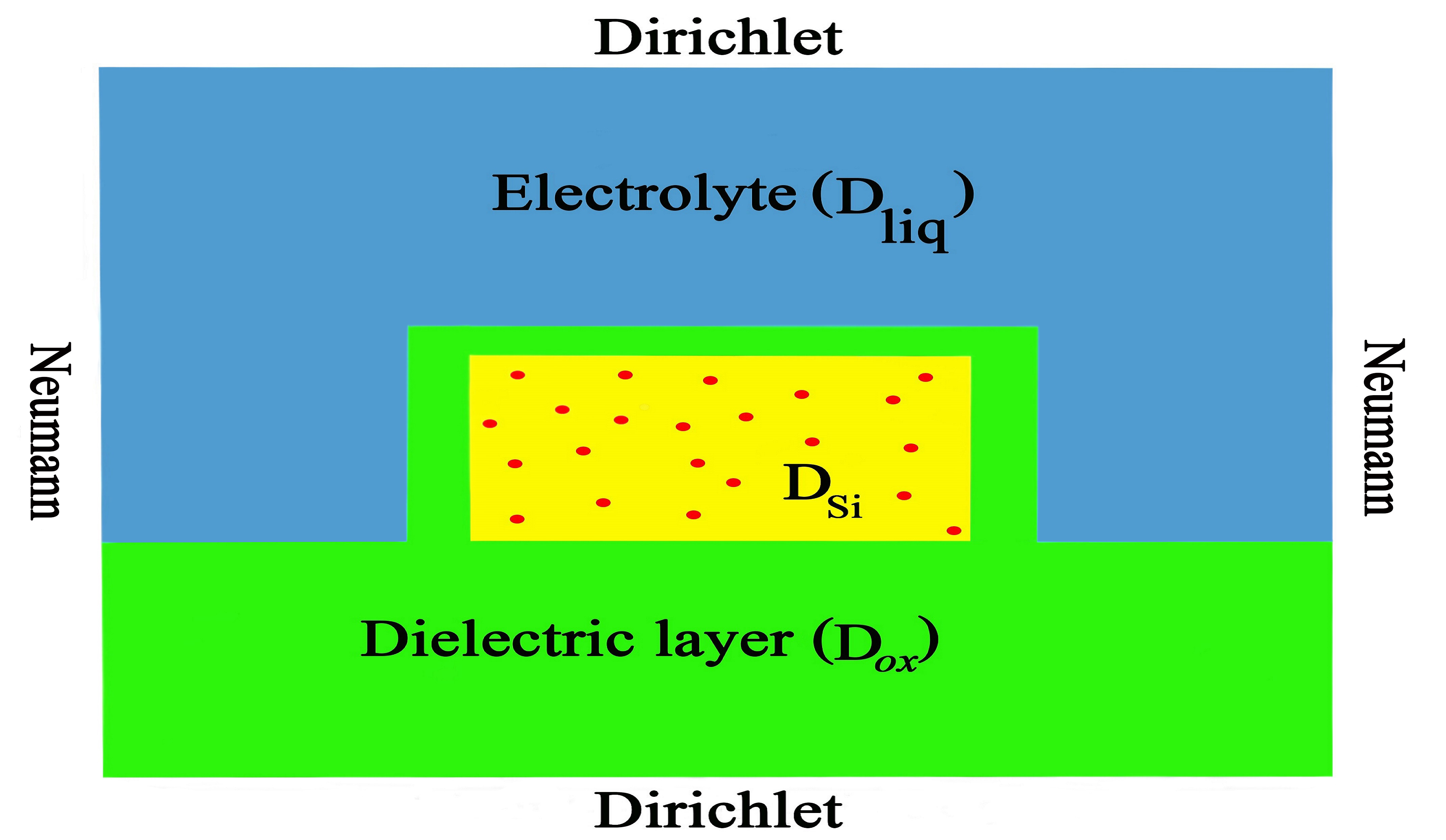}
  \caption{Schematic diagram showing leading applications.  In a
    field-effect transistor, the dopant atoms are randomly distributed
    (there is no electrolyte).  In field-effect sensors, there are
    randomly distributed dopant atoms as well as randomly distributed
    charged molecules in the electrolyte.}
  \label{f:schematic}
\end{figure}

In many realistic situations, the probability space is
high-dimensional.  For example, each probe molecule, each target
molecule, and each probe-target complex needs to be modeled in
sensors.  In transistors, the number impurities and their positions
are random.  The large number of dimensions favors the use of
Monte-Carlo (MC) methods: It is well-known that the convergence rate
of standard MC methods is independent of the number of dimensions.  On
the other hand, it is inversely proportional to the square root of the
number of evaluations and here each evaluation requires solving a two-
or three-dimensional system of elliptic equations.


These considerations motivate the development of a multi-level
Monte-Carlo (MLMC) algorithm.  In \cite{giles2008multilevel}, after
earlier work \cite{heinrich2001multilevel} on numerical quadrature, it
was shown that a multi-level approach and a geometric sequence of
timesteps can reduce the order of computational complexity of MC path
simulations for estimating the expected value of the solution of a
stochastic ordinary differential equation.  This is done by reducing
the variance and leaving the bias unchanged due to the Euler
discretization used as the ODE solver. In \cite{giles2008improved},
the Milstein scheme was used as the ODE solver to improve the
convergence rate of the MLMC method for scalar stochastic ordinary
differential equations and the method was made more efficient.  The
new method has the same weak order of convergence, but an improved
first-order strong convergence, and it is the strong order of
convergence which is central to the efficiency of MLMC methods. In
\cite{giles2009multilevel}, the MLMC method was combined with
quasi-Monte-Carlo (QMC) integration using a randomized rank-1 lattice
rule and the asymptotic order of convergence of MLMC was improved and
a lower computational cost was achieved as well.

In \cite{barth2011multi}, an MLMC finite-element method was presented
for elliptic partial differential equations with stochastic
coefficients. In this problem, the source of randomness lies in the
coefficients inside the operator and the coefficient fields are
bounded uniformly from above and away from zero.  The MLMC error and
work estimates were given for the expected values of the solutions and
for higher moments.  Also, in \cite{cliffe2011multilevel}, the same
problem was considered and numerical results indicate that the MLMC
estimator is not limited to smooth problems.  In \cite{kuo2012quasi},
a multi-level quasi-Monte-Carlo finite-element method for a class of
elliptic PDEs with random coefficients was presented.  The error
analysis of QMC was generalized to a multi-level scheme with the
number of QMC points dependent on the discretization level and with a
level-dependent dimension truncation strategy.

In \cite{charrier2013finite}, uniform bounds on the finite-element
error were shown in standard Bochner spaces. These new bounds can be
used to perform a rigorous analysis of the MLMC method for elliptic
problems, and a rigorous bound on the MLMC complexity in a more
general case was found.  In \cite{teckentrup2013further}, the
finite-element error analysis was extended for the same type of
equations posed on non-smooth domains and with discontinuities in the
coefficient. In \cite{haji2015optimization}, a general optimization of
the parameters in the MLMC discretization hierarchy based on uniform
discretization methods with general approximation orders and
computational costs was developed. In current work, we define a global
optimization problem which minimizes the computational complexity such
that the error bound is less or equal to a given tolerance level.

The rest of this paper is organized as follows. In
Section~\ref{s:model}, we present the system of model equations with
stochastic coefficients in detail.  In Section~\ref{s:existence}, we
define weak solutions of the model equations and prove existence and
local-uniqueness theorems.  Section~\ref{s:GFEM} collects results
about the FEM for later use.  In Section~\ref{s:MLMCFEM}, we introduce
a multi-level Monte-Carlo finite-element method for the system and
analyze its rate of convergence.  In Section~\ref{s:complexity}, we
discuss the computational complexity and find the optimal MLMC
method. In Section~\ref{s:implementation}, we present numerical
results for random impurity atoms in nanowire field-effect
sensors. The MLMC-FEM method is illustrated there and the
computational costs of various numerical techniques are compared as
well. Finally, conclusions are drawn in Section~\ref{conclusions}.
 
\section{The Stochastic Model Equations}\label{s:model}

Suppose that the domain $D \subset \R^d$ is bounded and convex, and
that $d \le 3$.  The whole domain~$D$ is partitioned into three
subdomains with different physical properties and hence different
model equations in order to include a large range of applications.
The first subdomain $\DSi$ consists of the (silicon) nanowire and acts
as the transducer of the sensor; in this subdomain, the
drift-diffusion-Poisson system describes charge transport. The
semiconductor is surrounded by a dielectric layer (usually an oxide)
which comprises the second subdomain $\Dox$, where just the Poisson
equation holds.  Finally, the third subdomain $\Dliq$ is the aqueous
solution containing cations and anions and the Poisson-Boltzmann
equation holds.  Also, the boundary layer at the sensor surface is
responsible for the recognition of the target molecules.  In the case
of field-effect sensors, solving a homogenization problem gives rise
to two interface conditions for the Poisson equation
\cite{heitzinger2010multiscale}.  In summary, the domain is
partitioned into
\begin{equation*}
  D = \DSi \cup \Dox \cup \Dliq. 
\end{equation*}
 
In the subdomain $\DSi$, the stationary drift-diffusion-Poisson system
\begin{subequations}
  \label{equation1}
  \begin{align}
    -\nabla \cdot(A(x,\omega)\nabla  V(x,\omega))
    &=q(\Cdop(x,\omega)+p(x,\omega)-n(x,\omega)),\\ 
    \nabla\cdot J_n(x,\omega) &= qR(n(x,\omega),p(x,\omega)),\\ 
    \nabla \cdot J_p(x,\omega) &= -qR(n(x,\omega),p(x,\omega)),\\ 
    J_n(x,\omega) &= q(D_n\nabla n(x,\omega)-\mu_nn(x,\omega)\nabla V(x,\omega)),\\ 
    J_p(x,\omega) &= q(-D_p\nabla p(x,\omega)-\mu_pp(x,\omega)\nabla V(x,\omega))
  \end{align}
\end{subequations}
models charge transport, where $A(x,\omega)$, the permittivity, is a
random field with $x \in \R^d$ and a random parameter
$\omega\in\Omega$ in a probability space $(\Omega, \mathbb{A},
\mathbb{P})$.  $\Omega$ denotes the set of elementary events, i.e.,
the sample space, $\mathbb{A}$ the $\sigma$-algebra of all possible
events, and $\mathbb{P}\colon \mathbb{A} \to [0,1]$ is a probability
measure.  $V(x,\omega)$ is the electrostatic potential and $q>0$ is
the elementary charge, $\Cdop(x,\omega)$ is the doping concentration,
$n(x,\omega)$ and $p(x,\omega)$ are the concentrations of electrons
and holes, respectively, $J_n(x,\omega)$ and $J_p(x,\omega)$ are the
current densities, $D_n$ and $D_p$ are the diffusion coefficients,
$\mu_n$ and $\mu_p$ are the mobilities, and
$R(n(x,\omega),p(x,\omega))$ is the recombination rate. We use the
Shockley-Read-Hall recombination rate
\begin{equation*}
  R(n(x,\omega),p(x,\omega))
  := \frac{n(x,\omega)p(x,\omega)-n_i^2}
  {\tau_p(n(x,\omega)+n_i)+\tau_n(p(x,\omega)+n_i)}
\end{equation*}
here, where the constant $n_i$ is the intrinsic charge density and
$\tau_n$ and $\tau_p$ are the lifetimes of the free carriers, although
the mathematical results here hold for many expressions for the
recombination rate.  Equations \eqref{equation1} include the
convection terms $-n\nabla V$ and $-p\nabla V$, which prohibit the use
of the maximum principle in a simple way.

We assume that the Einstein relations $D_n = U_T \mu_n$ and
$D_p = U_T \mu_p$ hold, where the constant $U_T$ is the thermal
voltage. Therefore, it is beneficial to change variables from the
concentrations~$n$ and~$p$ to the Slotboom variables~$u$ and~$v$
defined by
\begin{align*}
  n(x,\omega) &=: n_i \e^{V(x,\omega)/U_T}u(x,\omega),\\
  p(x,\omega) &=: n_i \e^{-V(x,\omega)/U_T}v(x,\omega).
\end{align*}
The system \eqref{equation1} then becomes
\begin{align*}
  -\nabla \cdot(A\nabla  V(x,\omega))
  &=qn_i(e^{-V(x,\omega)/U_T}v(x,\omega)-e^{V(x,\omega)/U_T}u(x,\omega))
    + q\Cdop(x,\omega),\\
  U_T\nabla\cdot(\mu_n e^{V(x,\omega)/U_T}\nabla u(x,\omega))
  &=\frac{u(x,\omega)v(x,\omega)-1}
    {\tau_p(e^{V(x,\omega)/U_T}u(x,\omega)+1)+\tau_n(e^{-V(x,\omega)/U_T}v(x,\omega)+1)},\\
  U_T\nabla\cdot(\mu_p e^{-V(x,\omega)/U_T}\nabla v(x,\omega))
  &=\frac{u(x,\omega)v(x,\omega)-1}
    {\tau_p(e^{V(x,\omega)/U_T}u(x,\omega)+1)+\tau_n(e^{-V(x,\omega)/U_T}v(x,\omega)+1)}, 
\end{align*}
where the continuity equations are self-adjoint.

The boundary $\partial D$ is partitioned into Dirichlet and Neumann
boundaries. For the Ohmic contacts we have
\begin{equation*}
  V(x,\omega)|_{\partial D_D}=V_D(x), \quad
  u(x,\omega)|_{\partial D_{\mathrm{Si}, D}}=u_D(x) \quad \text{and} \quad
  v(x,\omega)|_{\partial D_{\mathrm{Si}, D}}=v_D(x).
\end{equation*}
At Ohmic contacts the space charge vanishes, i.e., $\Cdop +p_D-n_D=
0$, and the system is in thermal equilibrium, i.e., $n_D p_D = n_i^2$
on $\partial D_D$. Furthermore, at each contact, the quasi Fermi
potential levels of silicon are aligned with an external applied
voltage $U(x)$. Therefore, by using the quasi Fermi potential, we
determine the boundary condition on $\partial D_{\mathrm{Si},D}$ using
\begin{equation*}
  V_1(x) :=
  U(x)+U_T\ln \left( \frac{n_D(x)}{n_i}\right) =
  U(x)-U_T\ln \left( \frac{p_D(x)}{n_i}\right).
\end{equation*}
The boundary values $u_D(x)$ and $v_D(x)$ are found to be
\begin{align*}
  u_D(x) &:= n_i^{-1}e^{-V_1(x)/U_T}n_D(x),\\
  v_D(x) &:= n_i^{-1}e^{V_1(x)/U_T}p_D(x),
\end{align*}
where
\begin{align*}
  n_D(x) &:= \frac{1}{2} \left(  \Cdop+\sqrt{\Cdop^2+4n_i^2}\right),\\
  p_D(x) &:= \frac{1}{2} \left(- \Cdop+\sqrt{\Cdop^2+4n_i^2}\right)
\end{align*}
hold \cite[Chapter~3]{Markowich1990semiconductor}. Here, $\Cdop :=
N_{D}^+ - N_{A}^-$ is the net doping concentration, where $N_{D}^+$
and $N_{A}^-$ are the donor and acceptor concentrations, respectively.

The zero Neumann boundary conditions
\begin{equation*}
  \mathbf{n} \cdot \nabla V(x,\omega)=0,\quad
  \mathbf{n} \cdot \nabla u(x,\omega)=0,\quad
  \mathbf{n} \cdot \nabla v(x,\omega)=0
\end{equation*}
hold on the rest of the boundary $\partial D$.  Here $\mathbf{n}$
denotes the unit outward normal vector on the boundary.

A jump in the permittivity~$A$ always gives rise to two continuity
conditions: the continuity of the potential and the continuity of the
electric displacement field.  Homogenization of an elliptic problem
with a periodic boundary layer at a manifold~$\Gamma$ yields the two
interface conditions \cite{heitzinger2010multiscale}
\begin{align*}
  V (0+,y,\omega)-V(0-,y,\omega) &= \alpha(y,\omega),\\
  A(0+)\partial_x V(0+,y,\omega)-A(0-)\partial_x V(0-,y,\omega)
                                 &= \gamma(y,\omega)
\end{align*}
between the semiconductor and the liquid.  Here we denote the
one-dimensional coordinate orthogonal to the manifold~$\Gamma$ by~$x$
and the remaining $(d-1)$-dimensional coordinates by~$y$.  $\alpha$
and $\gamma$ are essentially given by the dipole-moment and the
surface-charge densities of the boundary layer; in general, we write
them as the functional $M_\alpha(V)$ and $M_\gamma(V)$ of the
potential~$V$.  They may correspond to the Metropolis Monte-Carlo
method \cite{bulyha2011algorithm}, to solving the nonlinear
Poisson-Boltzmann equation \cite{heitzinger2010calculation}, or to
systems of ordinary differential equations for surface reactions
\cite{fort2006surface, Tulzer2013kinetic}.

In the subdomain $\Dox$, there are no charge carriers and the Poisson
equation is simply
\begin{equation*}
  -\nabla\cdot(A \nabla V(x,\omega))=0. 
\end{equation*}

In the subdomain $\Dliq$, the nonlinear Poisson-Boltzmann equation
\begin{equation*}
  -\nabla\cdot(A(x,\omega)\nabla V(x,\omega))
  + 2\eta\sinh(\beta(V(x,\omega) - \Phi(x,\omega)))=0
\end{equation*}
holds and models screening by free charges.  Here $\eta$ is the ionic
concentration, the constant $\beta$ equals $\beta := q/(k_BT)$ in
terms of the Boltzmann constant $k_B$ and the temperature $T$, and
$\Phi$ is the Fermi level.

In summary, for all $\omega\in \Omega$, the model equations are the
boundary-value problem
\begin{subequations}\label{modeleqn}
  \begin{alignat}{2}
   & -\nabla\cdot(A(x,\omega)\nabla V(x,\omega))\\\nonumber
   & = q \Cdop(x,\omega) - q n_i(\e^{V(x,\omega)/U_T}u(x,\omega)
    - \e^{-V(x,\omega)/U_T}v(x,\omega))
    && \text{in } \DSi,\\
   & -\nabla\cdot(A(x,\omega)\nabla V(x,\omega))
    = 0
    && \text{in } \Dox,\\
   & -\nabla\cdot(A(x,\omega)\nabla V(x,\omega))
    = -2\eta \sinh(\beta(V(x,\omega)-\Phi(x,\omega)))
    && \text{in } \Dliq,\\
   &  V(0+,y,\omega) - V(0-,y,\omega)
    = \alpha(y,\omega)
    && \text{on } \Gamma,\\
   & A(0+)\partial_x V(0+,y,\omega)-A(0-)\partial_x V(0-,y,\omega)
    = \gamma(y,\omega)
    && \text{on } \Gamma,\\
   & U_T\nabla\cdot(\mu_n \e^{V(x,\omega)/U_T}\nabla u(x,\omega))\\\nonumber
   & =\frac{u(x,\omega)v(x,\omega)-1}
    {\tau_p(\e^{V(x,\omega)/U_T}u(x,\omega)+1)
      +\tau_n(\e^{-V(x,\omega)/U_T}v(x,\omega)+1)}
    && \text{in } \DSi,\\
   & U_T\nabla\cdot(\mu_p \e^{-V(x,\omega)/U_T}\nabla v(x,\omega))\\\nonumber
   & =\frac{u(x,\omega)v(x,\omega)-1}
    {\tau_p(\e^{V(x,\omega)/U_T}u(x,\omega)+1)
      +\tau_n(\e^{-V(x,\omega)/U_T}v(x,\omega)+1)}
    && \text{in } \DSi,\\
   & \alpha(y,\omega) = M_{\alpha}(V(y,\omega))
    && \text{in } \Gamma,\\
   & \gamma(y,\omega) =M_{\gamma}(V(y,\omega))
    && \text{in } \Gamma,\\
   & V(x,\omega) = V_D(x)
    && \text{on } \partial D_D,\\
   & \mathbf{n} \cdot \nabla V(x,\omega) =0
    && \text{on } \partial D_N,\\
   & u(x,\omega) = u_D(x),\quad v(x,\omega) = v_D(x)
    && \text{on } \partial D_{D,\mathrm{Si}},\\
   & \mathbf{n} \cdot \nabla u(x,\omega)=0,\quad
    \mathbf{n} \cdot \nabla v(x,\omega)=0
    && \text{on }\partial D_{N,\mathrm{Si}}.
  \end{alignat}
\end{subequations}

\section{Existence and Local Uniqueness}
\label{s:existence}

In order to state the main theoretical results, we first record the
assumptions on the data of the system (\ref{modeleqn}).  The
assumptions are moderate in the sense that similar ones are necessary
for the deterministic system of equations.  Then weak solutions and
Bochner spaces are defined.  Using the assumptions and definitions,
existence and local uniqueness are shown.

\subsection{Assumptions}

The following assumptions are required.

\begin{assumptions}\label{assump1}
  \begin{enumerate}
  \item The bounded domain $D \subset \R^3$ has a $C^2$ Dirichlet
    boundary $\partial D_D$, the Neumann boundary $\partial D_N$
    consists of $C^2$ segments, and the Lebesgue measure of the
    Dirichlet boundary $\partial D_D$ is nonzero.  The $C^2$ manifold
    $\Gamma \subset D$ splits the domain~$D$ into two nonempty domains
    $D^+$ and $D^-$ so that
    $\mathop{meas}(\Gamma \cap \partial D) = 0$ and
    $\Gamma \cap \partial D \subset \partial D_N$ hold.

  \item $(\Omega, \mathbb{A}, \mathbb{P})$ is a probability space,
    where $\Omega$ denotes the set of elementary events (sample
    space), $\mathbb{A}$ the $\sigma$-algebra of all possible events,
    and $\mathbb{P}\colon \mathbb{A} \to [0,1]$ is a probability
    measure.

  \item The diffusion coefficient $A(x,\omega)$ is assumed to be a
    strongly measurable mapping from $\Omega$ into $L^\infty(D)$. It
    is uniformly elliptic and bounded function of position $x \in D$
    and the elementary event $\omega \in \Omega$, i.e., there exist
    constants $0 < A^- < A^+ < \infty$ such that
    \begin{equation*}
      0 < A^- \le
      \operatorname{ess\,inf}_{x \in D} A(x,\omega)
      \leq \| A(\cdot,\omega) \|_{L^\infty(D)}
      \leq A^+ < \infty
      \quad\forall\omega\in\Omega.
    \end{equation*}
    Furthermore, $A(x,\omega)|_{D^+ \times \Omega} \in C^1(D^+\times
    \Omega, \R^{3\times3})$ and $A(x,\omega)|_{D^-\times\Omega} \in
    C^1(D^-\times\Omega, \R^{3\times3})$.
    \label{assump A}

  \item The doping concentration $\Cdop(x,\omega)$ is bounded above
    and below with the bounds
    \begin{equation*}
      \underline{C}
      := \inf_{x \in D} \Cdop(x,\omega)
      \le \Cdop(x,\omega)\leq \sup_{x \in D} \Cdop(x,\omega)
      =: \overline{C} \quad \forall \omega \in \Omega.
    \end{equation*}


  \item There is a constant $\R \ni K \ge 1$ satisfying
    \begin{equation*}
      \frac{1}{K}\leq u_D(x),v_D(x) \le K
      \quad \forall x \in \partial D_{\mathrm{Si},D}.
    \end{equation*}

  \item The functionals
    $M_\alpha:  L^2(\Omega; H^1(D)) \cap L^\infty(D \times \Omega) \to
    L^2(\Omega; H^{1/2}(\Gamma)) \cap L^\infty(\Gamma \times \Omega)$
    and
    $M_\gamma:  L^2(\Omega; H^1(D)) \cap L^\infty(D \times \Omega) \to
    L^\infty(\Gamma \times \Omega)$ are continuous.
    \label{ass6}
    
  \item The mobilities $\mu_n(x,\omega)$ and $\mu_p(x,\omega)$ are
    uniformly bounded functions of $x \in D$ and $\omega \in \Omega$,
    i.e.,
    \begin{alignat*}{2}
      0<\mu_n^- &\leq \mu_n(x,\omega)\leq \mu_n^+<\infty
      &\qquad&\forall x \in D,\quad\forall\omega\in\Omega,\\
      0<\mu_p^- &\leq \mu_p(x,\omega)\leq \mu_p^+<\infty
      &\qquad&\forall x \in D,\quad\forall\omega\in\Omega,
    \end{alignat*}
    where $\mu_p(x,\omega), \mu_n(x,\omega) \in C^1(\DSi\times\Omega,
    \R^{3\times3})$.
    \label{ass7}
    \label{assump mu}
    
    Furthermore, the inclusions $f(x,\omega)\in L^2(\Omega;L^2(D))
    \cap L^{\infty} (D\times\Omega)$, $V_D(x) \in H^{1/2}(\partial D )
    \cap L^{\infty} (\Gamma)$, $u_D,v_D(x) \in H^{1/2}(\partial
    \DSi)$, $\alpha(x,\omega)\in L^2(\Omega; H^{1/2}(\Gamma))$, and
    $\gamma(x,\omega)\in L^2(\Omega; L^2(\Gamma))$ hold.
  \end{enumerate}
\end{assumptions}

Assumptions \ref{assump A} and \ref{assump mu} guarantee the uniform
ellipticity of the Poisson and the continuity equations, respectively.


\subsection{Weak Solution of the Model Equations}
\label{s:weaksol}

In order to define the weak formulation of the stochastic
boundary-value problem \eqref{modeleqn}, it suffices to consider the
semilinear boundary-value problem
\begin{subequations}\label{simplemodel}
  \begin{alignat}{3}\label{simplemodela}
    -\nabla\cdot(A(x,\omega)\nabla w(x,\omega))+h(x,w(x,\omega))
    &= f(x,\omega)
    &\quad&\forall x \in D \setminus \Gamma
    &\quad&\forall\omega \in \Omega,\\
    w(x,\omega)&=w_D(x)
    &\quad&\forall x \in \partial D_D
    &\quad&\forall\omega \in \Omega,\\
    \mathbf{n} \cdot \nabla w(x,\omega)&=0
    &\quad&\forall x \in \partial D_N
    &\quad&\forall \omega \in \Omega,\\
    w(0+,y,\omega) - w(0-,y,\omega) &= \alpha(y,\omega)
    &\quad&\forall x \in \Gamma
    &\quad&\forall \omega \in \Omega,\\
    A(0+)\partial_x w(0+,y,\omega)-A(0-)\partial_x w(0-,y,\omega)
    &= \gamma(y,\omega)
    &\quad&\forall x \in \Gamma
    &\quad&\forall\omega \in \Omega,
  \end{alignat}
\end{subequations}
which is a semilinear Poisson equation with interface conditions.  The
coefficient~$A$ here is either equal to~$A$ or equal to $\mu_n
\e^{V/U_T}$ or $\mu_p \e^{-V/U_T}$ in \eqref{modeleqn}.  However,
uniform ellipticity holds in each of these cases per
Assumption~\ref{assump1}.

For the weak formulation, we define the Hilbert space
\begin{equation}\label{sol space}
  X := H^1_g(D) = \left\{ w \in H^1(D) \mid Tw = g \right\}
\end{equation}
as the solution space, where $T$ is the trace operator defined such
that $Tw=g$, where $g$ is Dirichlet lift of $w_D:=w|_{\partial D_D}$.
The operator~$T$ is well-defined and continuous from $H^1(D)$ onto
$H^{1/2}(\partial D)$ for the Lipschitz domain~$D$.  For $g=0$, we
define the test space
\begin{equation}\label{test space}
  X_0 := H^1_0(D) = \left\{w \in H^1(D) \mid Tw=0\right\}.
\end{equation}


\begin{definition}[Bochner spaces]\label{def0}
  Given a Banach space $(X, \|\cdot\|_X)$ and $1 \le p \leq +\infty$,
  the \emph{Bochner space} $L^p(\Omega;X)$ is defined to be the space
  of all measurable functions $w\colon \Omega \to X$ such that for
  every $\omega\in \Omega$ the norm
  \begin{align}\label{bochner norm}
    \| w \|_{L^p(\Omega;X)} :=
    \begin{cases}
      \Big( \int_\Omega \| w(\cdot,\omega) \|_X^p \rmd\P(\omega)
      \Big)^{1/p}
      = \E\Big[\| w(\cdot,\omega) \|_X^p \Big]^{1/p}
      < \infty, & 1 \leq p < \infty,\\
      \operatorname*{ess\,sup}_{\omega \in \Omega}
      \| w(\cdot,\omega) \|_X < \infty, & p = \infty
    \end{cases}
  \end{align}
  is finite.
\end{definition}

To derive the variational formulation of our model
\eqref{simplemodel}, we fix the event $\omega \in \Omega$ at first,
multiply ({\ref{simplemodela}}) by a test function $\phi \in L^2(
\Omega; X_0)$, and integrate by parts in~$D$ to obtain the relation
\begin{equation*}
  \int_D A\nabla w\cdot \nabla \phi+ \int_D h(w)\phi
  = \int_Df\phi + \int_{\Gamma} \gamma \phi
  \qquad\forall \phi \in L^2(\Omega;X_0).
\end{equation*}

\begin{definition}[Weak solution on $D \times \Omega$]\label{def2}
  Suppose that $A$ satisfies Assumptions~\ref{assump1} and that
  $f(x,\omega) \in L^2(\Omega; L^2(D))$, $w_D(x) \in H^{1/2}(\partial
  D_D)$, and $\gamma(x,\omega) \in L^2(\Omega;L^2(\Gamma))$ holds. A
  function~$w \in L^2(\Omega; X)$ is called a \emph{weak solution} of
  the boundary-value problem ({\ref{simplemodel}}), if it satisfies
  \begin{equation}\label{weaksol}
    a(w,\phi)=\ell(\phi)\qquad\forall \phi \in L^2(\Omega;X_0),
  \end{equation}
  where $a\colon L^2(\Omega;X) \times L^2(\Omega;X_0)\to \R$ and
  $\ell\colon L^2(\Omega; X_0)\to \R$ are defined by
  \begin{equation*}
    a(w,\phi)
    := \E\left[ \int_D A\nabla w\cdot\nabla \phi\rmd x\right]
    + \E\left[\int_D
      h\left(w\right)
      \phi \rmd x \right]
  \end{equation*}
  and
  \begin{equation*}
    \ell(\phi) := \E\left[ \int_D f \phi \rmd x\right] +
    \E\left[\int_\Gamma \gamma \phi \rmd x \right].
  \end{equation*}
\end{definition}

\subsection{Existence and Local Uniqueness of the Solution}
\label{s:main existence}

In the next step, we prove existence and local uniqueness of solutions
of system of stochastic elliptic boundary-value problems with
interface conditions \eqref{modeleqn} using the Schauder fixed-point
theorem and the implicit-function theorem similarly to
\cite[Theorem 2.2 and 5.2]{baumgartner2012existence}.

\begin{theorem}[Existence]\label{existence}
  Under Assumptions~\ref{assump1}, for every $f(x,\omega) \in
  L^2(\Omega; L^2(D))$ and $V_D, u_D, v_D\in H^{1/2}(\partial D)$,
  there exists a weak solution
  \begin{multline*}
    (V(x,\omega),u(x,\omega),v(x,\omega),\alpha(x,\omega),
    \gamma(x,\omega))
    \in \big( L^2(\Omega;H^1_{V_D}(D) \cap L^\infty(D\times\Omega)\big)\\
    {} \times \big( L^2(\Omega;H^1_{u_D}(\DSi))
    \cap L^\infty(\DSi\times\Omega)\big)
    \times \big(L^2(\Omega;H^1_{v_D}(\DSi))
    \cap L^\infty(\DSi\times\Omega)\big)\\
    \times \big(L^2(\Omega;H^1(\Gamma))\cap L^\infty(\Gamma\times\Omega)\big)^2
  \end{multline*}
  of the stochastic boundary-value problem (\ref{modeleqn}), and for
  every $\omega\in \Omega$ it satisfies the $L^\infty$-estimate
  \begin{alignat*}{2}
    \underline{V}&\leq V(x,\omega)\leq \overline{V} &\qquad&\text{in } D,\\
    \frac{1}{K}  &\leq u(x,\omega)\leq K            &      &\text{in } \DSi,\\
    \frac{1}{K}  &\leq v(x,\omega)\leq K            &      &\text{in } \DSi,
  \end{alignat*}
  where
  \begin{align*}
    \underline{V}&:=\min(\inf_{\partial D_D}V_D, \Phi-\sup_{D}V_L,
                   U_T\ln(\frac{1}{2Kn_i}
                   (\underline{C}+\sqrt{\underline{C}^2 + 4n_i^2}))-\sup_{D}V_L),\\
    \overline{V} &:=\max(\sup_{\partial D_D}V_D, \Phi-\inf_{D}V_L,
                   U_T\ln(\frac{K}{2n_i}
                   (\overline{C}+\sqrt{\overline{C}^2 + 4n_i^2}))-\inf_{D}V_L).
  \end{align*}
  Here $V_L(x,\omega)$ is the solution of the linear problem (i.e.,
  problem \eqref{simplemodel} with $h\equiv0$), for which the estimate
  \begin{equation*}
    \|  V_L \|_{ L^2(\Omega;H^1_{V_D}(D))}
    \leq C \left( \|f \|_{L^2(\Omega; L^2(D))}
      +\| V_D \|_{ H^{1/2}({\partial D_D})}
      +\| \alpha \|_{ L^2(\Omega; H^{1/2}(\Gamma))}\\
      +\| \gamma \|_{ L^2(\Omega; L^{2}(\Gamma))} \right)
  \end{equation*}
  holds, where $C$ is a positive constant. 
\end{theorem}

\begin{proof}
  The existence of the solution is proved using the Schauder
  fixed-point theorem and the estimates are obtained from a maximum
  principle. First, we define a suitable space
  \begin{multline*}
    N := \bigl\{  (V,u,v,\alpha, \gamma) \in L^2(\Omega; H^1(D))
    \times  L^2(\Omega;H^1({D_\mathrm{Si}}))^2 
    \times L^2(\Omega; H^1(\Gamma))^2 \bigm|\\
    \underline{V}\leq V(x,\omega)\leq \overline{V}
    \quad\text{a.e.\ in } D \times\Omega,
    \quad\frac{1}{K}\leq u(x,\omega), v(x,\omega)\leq K
    \quad\text{a.e.\ in } \DSi \times \Omega,\\
    \alpha, \gamma\text{ bounded a.e.\ on } \Gamma\times \Omega \bigr\},
  \end{multline*}
  which is closed and convex. Then we define a fixed-point
  map $F\colon N \to N $ by
  \begin{equation*}
    F(V_0,u_0,v_0,\alpha_0, \gamma_0)
    := (V_1,u_1,v_1,\alpha_1,\gamma_1),
  \end{equation*}
  where the elements of the vector $(V_1,u_1,v_1,\alpha_1,\gamma_1)$
  are the solutions of the following equations for given data
  $(V_0,u_0,v_0,\alpha_0,\gamma_0)$.
  \begin{enumerate}
  \item Solve the elliptic equation
      \begin{alignat*}{3}
    	-\nabla\cdot(A\nabla  V_1)
        &=qn_i(e^{-V_1/U_T}v_0-e^{V_1/U_T}u_0)+q\Cdop \quad&&\text{in}~D,\\
    	\mathbf{n}\cdot\nabla V_1&=0&&\text{on}~\partial D_N,\\
        V_1&=V_D&&\text{on}~\partial D_D
      \end{alignat*}
    for $V_1$.
  \item Solve the elliptic equation
    \begin{alignat*}{3}
      U_T\nabla \cdot(\mu_n e^{V_1/U_T}\nabla u_1)\\
      -\frac{u_1v_0-1}{\tau_p(e^{V_1/U_T}u_0+1)+\tau_n(e^{-V_1/U_T}v_0+1)}&=0\quad&&\text{in}~\DSi,\\  
      \mathbf{n}\cdot\nabla u_1&=0&&\text{on}~\partial D_{\mathrm{Si},N},\\
      u_1&=u_D&&\text{on}~\partial D_{\mathrm{Si},D}
    \end{alignat*}
    for $u_1$.
  \item Solve the elliptic equation
    \begin{alignat*}{3}
      U_T\nabla\cdot(\mu_p e^{-V_1/U_T}\nabla v_1)\\
      -\frac{u_0v_1-1}{\tau_p(e^{V_1/U_T}u_0+1)+\tau_n(e^{-V_1/U_T}v_0+1)}&=0\quad&&\text{in}~\DSi\\  
      \mathbf{n}\cdot\nabla v_1&=0&&\mathrm{on}~\partial D_{\mathrm{Si},N},\\
      v_1&=v_D&&\mathrm{on}~\partial D_{\mathrm{Si},D},
    \end{alignat*}
    for $v_1$. 
  \item Update the surface-charge density and dipole-moment density
    according to the microscopic model
    \begin{align*}
      \alpha_1(y,\omega) &:=M_\alpha(V_0),\\
      \gamma_1(y,\omega) &:=M_\gamma(V_0).
    \end{align*} 
  \end{enumerate} 

  Using Lemmata on the existence and uniqueness of solutions of
  elliptic boundary-value problems with interface conditions, every
  equation present in the model is uniquely solvable.  Therefore the
  map~$F$ is well-defined. Furthermore, continuity and the
  self-mapping property of~$F$ as well as the precompactness of $F(N)$
  can be shown similarly to \cite[Theorem
  2.2]{baumgartner2012existence} and
  \cite[Theorem~1]{Heitzinger2015existence}.  Therefore, applying the
  Schauder fixed-point theorem yields a fixed-point of~$F$, which is a
  weak solution of \eqref{modeleqn}.
\end{proof}

In general, the solution in Theorem~\ref{existence} is not unique;
uniqueness of the solution only holds in a neighborhood around thermal
equilibrium.  This necessitates sufficiently small Dirichlet boundary
conditions.  The following theorem yields local uniqueness of the
solution of our system \eqref{modeleqn} of model equation. The proof
is based on the implicit-function theorem.

\begin{theorem}[Local uniqueness]\label{uniqueness}
  Under Assumption~\ref{assump1}, for every $f(x,\omega) \in
  L^2(\Omega; L^2(D))$, $V_D, u_D, v_D\in H^{1/2}(\partial D)$,
  $\alpha\in L^2(\Omega; H^{1/2}(\Gamma))$, and $\gamma\in L^2(\Omega;
  L^2(\Gamma))$, there exists a sufficiently small $\sigma \in
  \mathbb{R}$ with $|U|<\sigma$ such that the stochastic problem in
  the existence theorem~\ref{existence} has a locally unique solution
  \begin{align*}
    \Big(V^{\ast}(U), u^{\ast}(U), v^{\ast}(U), \alpha^{\ast}(U), \gamma^{\ast}(U)\Big)
    &\in L^2(\Omega;H^2(D \setminus \Gamma))
      \times L^2(\Omega;H^2(D_{\mathrm{Si}}))^2\\
    &\times L^2(\Omega;H^{1/2}(\Gamma))
      \times L^2(\Omega;L^{2}(\Gamma)).
  \end{align*}
  The solution satisfies
  \begin{align*}
    \Big(V^{\ast}(0), u^{\ast}(0), v^{\ast}(0), \alpha^{\ast}(0), \gamma^{\ast}(0)\Big)
    = (V_e, 1, 1, \alpha_e, \gamma_e)
  \end{align*}
  and it depends continuously differentiably on~$U$ as a map from $\{
  U \in \mathbb{R}^k,\; |U|< \sigma \} $ into $ L^2(\Omega;H^2(D
  \setminus \Gamma)) \times L^2(\Omega;H^2(D_{\mathrm{Si}}))^2 \times
  L^2(\Omega;H^{1/2}(\Gamma)) \times L^2(\Omega;L^{2}(\Gamma))$.
\end{theorem}

\begin{proof}
  We call the equilibrium potential $V_e(x,\omega)$ and the
  equilibrium surface densities $\alpha_e(x,\omega)$ and
  $\gamma_e(x,\omega)$. $(V_e, 1, 1, \alpha_e, \gamma_e)$ is a
  solution of the stochastic equilibrium boundary-value problem, which
  has a unique solution due to the existence and uniqueness of
  solutions of stochastic semilinear elliptic boundary-value problems
  of the form
  \begin{alignat*}{2}
    -\nabla\cdot(A(x,\omega)\nabla  V_e(x,\omega)) &=
    q\Cdop(x,\omega)-qn_i(e^{V_e(x,\omega)/U_T}-e^{-V_e(x,\omega)/U_T})
    &\quad&\text{in } \DSi,\\
    -\nabla\cdot(A(x,\omega)\nabla  V_e(x,\omega)) &= 0
    &&\text{in } \Dox,\\
    -\nabla\cdot(A(x,\omega)\nabla  V_e(x,\omega)) &=
    -2\eta \sinh (\beta(V_e(x,\omega)-\Phi(x,\omega)))
    &&\text{in } D_{\text{liq}},\\
    V_e (0+,y,\omega)-V_e (0-,y,\omega) &= \alpha_e(y,\omega)
    &&\text{on } \Gamma,\\
    A(0+)\partial_x V_e (0+,y,\omega)
    &-A(0-)\partial_x V_e(0-,y,\omega)
    = \gamma_e(y,\omega)
    &&\text{on } \Gamma,\\
    V_e(x,\omega) &= V_D(x)
    &&\text{on } \partial D_D,\\
    \mathbf{n}\cdot\nabla V_e(x,\omega) &=0
    &&\text{on } \partial D_N. 
  \end{alignat*}
  To apply the implicit-function theorem, we define the map
  \begin{align*}
    & G \colon B \times S_{\sigma_1(0)} \to L^2(\Omega;L^2(D))\times L^2(\Omega;L^2(\DSi))^2 
      \times L^2(\Omega;L^2(\Gamma))^2,\\
    & G (V,u,v,\alpha, \gamma,U)=0,
  \end{align*}
  where $G$ is given by the boundary-value problem \eqref{modeleqn} after substituting
  $\overline{V} := V - V_D (U )$, $ \overline{u} := u-u_D (U )$, and
  $\overline{v} := v - v_D (U )$.  $B$ is an open subset of
  $L^2(\Omega;H^2_{\partial}(D))\times
  L^2(\Omega;H^2_{\partial}(D_{\rm Si}))^2 \times
  L^2(\Omega;L^2(\Gamma))^2$ with
  \begin{align*}
    H^2_{\partial}(D):=\{ \phi \in H^2(D)
    \mid \mathbf{n} \cdot \nabla \phi=0 \text{ on } \partial D_N, \;
    \phi=0 \text{ on } \partial D_D \},
  \end{align*} 
  and the sphere $S_{\sigma_1}$ with radius $\sigma_1$ and center~$0$
  is a subset of $\mathbb{R}^d$. The equilibrium solution
  $(V_e-V_D(0),0,0,\alpha_e,\gamma_e,0)$ is a solution of the equation
  $G=0$. One can show that the Fr\'{e}chet derivative
  $D_{(V,u,v,\alpha,\gamma)}G(V_e-V_D(0),0,0,\alpha_e,\gamma_e,0)$ has
  a bounded inverse (see, e.g.,
  \cite[Theorem~2.2]{baumgartner2012existence}). Then the
  implicit-function theorem implies uniqueness of the solution of
  \eqref{modeleqn}.
\end{proof}

\section{Multi-Level Monte-Carlo Finite-Element Method}
\label{s:MLMCFEM}

We start by briefly recapitulating the finite-element approximation of
the system of model equations considered here. Then we review the
types of error in the Monte-Carlo approximation of solutions of
stochastic partial differential equations in Section~\ref{s:MC-FEM}.
In Section~\ref{s:MLMC-FEM}, a multi-level Monte-Carlo (MLMC)
finite-element (FE) method for the solution of the system of
stochastic equations \eqref{modeleqn} is developed. We give an error
bound for MLMC-FEM approximation and discuss the computational
complexity.

\subsection{The Finite-Element Method}\label{s:GFEM}

In this subsection, we briefly recapitulate the Galerkin finite-element
approximation and fix some notation. It provides the foundation for
the following section.

We suppose that the domain~$D$ can be partitioned into quasi-uniform
triangles or tetrahedra such that sequences $\{\tau_{h_{\ell}}
\}_{{\ell}=0}^{\infty}$ of regular meshes are obtained. For any $ \ell
\geq 0$, we denote the mesh size of $\tau_{h_{\ell}}$ by
\begin{equation*}
  h_{\ell}
  := \max_{K\in \tau_{h_\ell}}\{ \operatorname{diam} K \}.
\end{equation*}
To ensure that the mesh quality does not deteriorate as refinements
are made, shape-regular meshes can be used.

\begin{definition}[Shape regular mesh]\label{def3}
  A sequence $\{\tau_{h_{\ell}} \}_{l=0}^{\infty}$ of meshes is
  \emph{shape regular} if there exists a constant $\kappa<\infty$
  independent of~$\ell$ such that
  \begin{equation*}
    \frac{h_K}{\rho_K}\geq \kappa\quad \forall K \in \tau_{h_\ell}.
  \end{equation*}
  Here $\rho_K$ is the radius of the largest ball that can be
  inscribed into any $K \in \tau_{h_{\ell}}$.
\end{definition}

Uniform refinement of the mesh can be achieved by regular subdivision.
This results in the mesh size
\begin{equation}\label{h_ell}
  h_\ell = r^{-\ell} h_0,
\end{equation}
where $h_0$ denotes the mesh size of the coarsest triangulation and
$r>1$ is independent of~$\ell$. The nested family
$\{\tau_{h_{\ell}} \}_{\ell=0}^{\infty} $ of regular triangulations
obtained in this way is shape regular.

The Galerkin approximation is the discrete version of the weak
formulation in (\ref{weaksol}) of the stochastic elliptic
boundary-value problem (\ref{modeleqn}). We consider finite-element
discretizations with approximations $u_h \in X_{h_\ell}$ of $u \in X$.
Given a mesh $\tau_{h_\ell}$, $X$ is the solution space \eqref{sol
  space} and $X_{h_\ell}\subset X$ is the discretized space. For
all $k \geq 1$, it is defined as
\begin{equation}\label{discr. space}
  X_{h_\ell}
  := \mathbb{P}^k(\tau_{h_{\ell}})
  := \{u\in X  \;\mid\;
  u \vert_K\in \mathbb{P}^k(K) \;\; \forall K\in \tau_{h_{\ell}} \},
\end{equation}
where ${\mathbb{P}}^k(K) := \operatorname{span}\{ x^{\alpha} \mid
|\alpha| \leq k\}$ is the space of polynomials of total degree less
equal~$k$.
The space $X_0$ is the space \eqref{test space} of test functions. The
discretized test space ${X_0}_{h_\ell}\subset X_0$ is defined
analogously to \eqref{discr. space}.

After introducing the finite-element spaces, everything is ready to
define the Galerkin approximation.

\begin{definition}[Galerkin approximation]\label{def4}
  Suppose $X_{h_\ell}\subset X$ and ${X_0}_{h_\ell}\subset X_0$.  The
  \emph{Galerkin approximation} of
  \eqref{simplemodel}
  is the function
  \begin{equation*}
    w_{h_\ell}\in L^2(\Omega;X_{h_\ell})
  \end{equation*}
  that satisfies
  \begin{equation}\label{discweaksol}
    B(w_{h_\ell},\phi_{h_\ell}) = F(\phi_{h_\ell})\qquad
    \forall \phi_{h_\ell} \in L^2(\Omega; {X_0}_{h_\ell}),
  \end{equation}
  where $B$ and $F$ are defined in (\ref{weaksol}).
\end{definition}

\subsection{Monte-Carlo Finite-Element Approximation}\label{s:MC-FEM}

The straightforward Monte-Carlo method for a stochastic PDE
approximates the expectation $\mathbb{E}[u]$ of the solution~$u$ by
the sample mean of a (large) number of evaluations. Since we use the
same finite-element mesh~$\tau$ with the mesh size~$h$ for all
samples, we drop the index~$\ell$ in this subsection for the MC-FEM.
We approximate $\mathbb{E}[u]$ by $\mathbb{E}[u_h]$, where $u_h$ is
again the FE approximation of~$u$ using a mesh of size~$h$. The
standard MC estimator $\EMC$ for $\mathbb{E}[u_h]$ is the sample mean
\begin{equation}\label{MC-estimator}
  \EMC[u_h]
  := \hat{u}_h
  := \frac{1}{M} \sum_{i=1}^{M} u_h^{(i)}.
\end{equation}
where $u_h^{(i)} = u_h(x,\omega^{(i)})$ is the $i$th
sample of the solution.

The following lemma shows the error of the MC estimator for a random
variable~$u$ which is not descretized in space is of order
$O(M^{-1/2})$.

\begin{lemma}\label{lemma:general MC error}
  For any number of samples $M \in \mathbb{N}$ and for a random variable
  $u \in L^2(\Omega;X)$, the inequality
  \begin{equation}\label{general MC error}
    \| {\mathbb{E}}[u]-\EMC[u] \|_{L^2(\Omega;X)}
    = M^{-1/2} \sigma[u]
  \end{equation}
  holds for the MC error, where $\sigma[u] :=
  \|\E[u]-u\|_{L^2(\Omega;X)}$.
\end{lemma}

\begin{proof}
  The result follows from the calculation
  \begin{align*}
    \| {\mathbb{E}}[u]-\EMC[u] \|^2_{L^2(\Omega;X)}
    &=\E\Big[ \Big\| \E[u] - \frac{1}{M}\sum_{i=1}^{M}u^{(i)}  \Big\|_X^2\Big]\\ 
    &=\frac{1}{M^2}\sum_{i=1}^{M}\E\Big[ \| \E[u] - u^{(i)}  \|_X^2\Big]\\ 
    &=\frac{1}{M}\E\Big[ \| \E[u] - u  \|_X^2\Big]
      =M^{-1}\sigma^2[u].
  \end{align*}
\end{proof}

Next we generalize the result to the finite-element solution by using
the MC estimator to approximate the expectation $\mathbb{E}[u]$ of a
solution~$u$ of an SPDE, which is descretized in space by the
finite-element method. In other words, if $u_h$ and $\hat{u}_h$ are
the finite-element and MC solutions of the SPDE, respectively, then we
have
\begin{equation*}
  \mathbb{E}[u] \approx \mathbb{E}[u_h] \approx \hat{u}_h.
\end{equation*}
Therefore, the MC-FEM method involves two approximations and hence
there are two sources of error.
\begin{description}
\item[Discretization error] The approximation of $\mathbb{E}[u]$ by
  $\mathbb{E}[u_h]$ gives to the discretization error, which stems
  from the spatial discretization.
\item[Statistical error] The approximation of the expected value
  $\mathbb{E}[u_h]$ by the sample mean $\hat{u}_h$ gives rise to the
  statistical error, which is caused by the MC estimator.
\end{description}

Lemma \ref{lemma:general MC error} takes care of the statistical
error. The order of the discretization error depends on the order of
the finite-element method.

\begin{proposition}\label{prop2}
  Suppose $\alpha, C_0, C_1 \in \R^+$. Let $\EMC[u_h]$ be the
  Monte-Carlo estimator with~$M$ samples to approximate the
  expectation $\mathbb{E}[u]$ of a solution $u(\cdot,\omega) \in X$ of
  an SPDE by using a FE solution $u_h(\cdot,\omega) \in X_h$ with mesh
  size~$h$. Suppose that the discretization error converges with order
  $\alpha$, i.e.,
  \begin{equation}\label{MC:assump1}
    \|\mathbb{E}[u-u_h]\|_X \le C_1 h^\alpha,
  \end{equation}
  and that the estimate
  \begin{equation}\label{MC:assump2}
    \sigma[u_h]\le C_0
  \end{equation}   
  holds. Then the error of the MC estimator satisfies
  \begin{equation}\label{MC error}
    \| \E[u]  - \EMC[u_h] \|_{L^2(\Omega;X)}
    = O(h^\alpha)+O(M^{-1/2}).
  \end{equation}
\end{proposition}

\begin{proof}
  We use the root mean square error (RMSE) to measure the accuracy of
  the total approximation and analyze the contribution of both errors.
  Using the triangle inequality and Lemma~\eqref{lemma:general MC
    error}, we find
  \begin{equation}\label{MC-error1}
    \begin{split}
      \RMSE
      &:=   \| \E[u]  - \EMC[u_h]  \|_{L^2(\Omega;X)} \\
      &\le  \| \E[u]  - \E[u_h]  \|_X +  \| \E[u_h]  - \EMC[u_h]  \|_{L^2(\Omega;X)}\\
      &\le \| \E[u-u_h] \|_X  +  M^{-1/2} \sigma[u_h]\\
      &\le C_1h^\alpha + C_0  M^{-1/2}\\
      &= O(h^\alpha)+O(M^{-1/2}).
    \end{split}
  \end{equation}
  The last expressions are obtained by using the assumtions
  \eqref{MC:assump1} and \eqref{MC:assump2}.
\end{proof}

\subsection{Multi-Level Monte-Carlo Finite-Element Approximation}
\label{s:MLMC-FEM}

In this section, we first present the MLMC FE method and an its error.
In this method, several levels of meshes are used and the MC estimator
is employed to approximate the solution on each level independently.
We start by discretizing the variational formulation \eqref{weaksol}
on the sequence
\begin{equation*}
  X_{h_0} \subset X_{h_1} \subset \cdots \subset
  X_{h_L} \subset X
\end{equation*}
of finite-dimensional sub-spaces, where $X_{h_\ell} :=
{\mathbb{P}}^1(\tau_{h_\ell})$ for all $\ell \in \{0,1,2,\ldots,L\}$
(see Section~\ref{s:GFEM}). The finite-element approximation at
level~$L$ can be written as the telescopic sum
\begin{equation*}
  u_{h_L}=u_{h_0}+\sum_{\ell=1}^{L}(u_{h_\ell}-u_{h_{{\ell}-1}}),
\end{equation*}
where each $u_{h_{\ell}}$ is the solution on the mesh
$\tau_{h_{\ell}}$ at level~$\ell$. Therefore, the expected value of
$u_{h_L}$ is given by
\begin{equation}\label{MLMC}
  {\mathbb{E}}[u_{h_L}]
  ={\mathbb{E}}[u_{h_0}]+{\mathbb{E}}\left[\sum_{{\ell}=1}^{L}(u_{h_{\ell}}-u_{h_{{\ell}-1}})\right]
  ={\mathbb{E}}[u_{h_0}]+\sum_{{\ell}=1}^{L}{\mathbb{E}}[u_{h_{\ell}}-u_{h_{{\ell}-1}}].
\end{equation}
In the MLMC FEM, we estimate
${\mathbb{E}}[u_{h_{\ell}}-u_{h_{{\ell}-1}}]$ by a level dependent
number $M_{\ell}$ of samples. The MLMC estimator ${\mathbb{E}}[u]$ is
defined as
\begin{equation}\label{MLMC-estimator}
  \EMLMC[u]
  :=\hat{u}_{h_L}
  :=\EMC[u_{h_0}]+\sum_{{\ell}=1}^{L} \EMC[u_{h_{\ell}}-u_{h_{{\ell}-1}}],
\end{equation}
where $E_\mathrm{MC}$ is the Monte-Carlo estimator defined in
(\ref{MC-estimator}). Therefore, we find
\begin{equation}
  \hat{u}_{h_L}
  =\frac{1}{M_0}\sum_{i=1}^{M_0}u_{h_0}^{(i)}+\sum_{{\ell}=1}^{L}\frac{1}{M_{\ell}}\sum_{i=1}^{M_{\ell}}(u_{h_{\ell}}^{(i)}-u_{h_{{\ell}-1}}^{(i)}).
\end{equation}

It is important to note that the approximate solutions
$u_{h_{\ell}}^{(i)}$ and $u_{h_{\ell-1}}^{(i)}$ correspond to the same
sample~$i$, but are computed on different levels of the mesh, i.e., on
the meshes $M_{\ell}$ and $M_{{\ell}-1}$, respectively.

Recalling the two sources of error constituting the MC-FE error, the
following result holds for the MLMC-FEM error.

\begin{proposition}\label{prop3}
  Suppose $\alpha, \beta, C_0, C_1 \in \R^+$.  Let $\EMLMC[u_{h_L}]$ be
  the multi-level Monte-Carlo estimator to approximate the expectation
  $\mathbb{E}[u]$ of a solution $u(\cdot,\omega)\in X$ of an SPDE
  by using a FE solution $u_{h_\ell}(\cdot,\omega)\in X_{h_\ell}$ with
  $M_\ell$ samples in level $\ell$, $\ell \in \{0,1,2,\ldots,L\}$ and with
  mesh size~$h_\ell$.  Suppose that the convergence order~$\alpha$
  for the discretization error, i.e.,
  \begin{equation}\label{MLMC:assump1}
    \|\mathbb{E}[u - u_{h_\ell}]\|_X \le C_1 h_\ell^\alpha,
  \end{equation}
  the convergence order~$\beta$ for 
  \begin{equation}\label{MLMC:assump2}
    \sigma[u_{h_\ell} - u_{h_{\ell-1}}] \le C_0 h_{\ell-1}^\beta,
  \end{equation}
  and assume that the estimate
  \begin{equation}\label{MLMC:assump3}
    \sigma[u_{h_0}] \le C_{00}
  \end{equation}
  holds, where $\sigma[u]:=\|\E[u]-u\|_{L^2(\Omega;X)}$. Then the
  error of the MLMC estimator satisfies
  \begin{equation}\label{MLMC error}
    \| \mathbb{E}[u] - \EMLMC[u_{h_L}] \|_{L^2(\Omega;X)}
    =O(h_L^\alpha) + O(M_0^{-1/2}) + \sum_{\ell=1}^{L} O(M_\ell^{-1/2})O(h_{\ell-1}^\beta).
  \end{equation}
\end{proposition}

\begin{proof}
  Similarly to the MC estimator, the RMSE assesses the accuracy of the
  MLMC FE estimator.  Using the triangle inequality and the relations
  \eqref{MLMC} and \eqref{MLMC-estimator}, we find
  \begin{align*}
    \RMSE
    &:=   \| \E[u]  - \EMLMC[u_{h_L}]  \|_{L^2(\Omega;X)} \\
    &\le  \| \E[u]  - \E[u_{h_L}]  \|_X +  \| \E[u_{h_L}]  - \EMLMC[u_{h_L}]  \|_{L^2(\Omega;X)}\\
    &\le \| \E[u-u_{h_L}] \|_X  +  \| \E[u_{h_0}]  - \EMC[u_{h_0}]  \|_{L^2(\Omega;X)} \\
    &\quad + \sum_{\ell=1}^{L}  \| \E[u_{h_\ell} - u_{h_{\ell-1}}]  - \EMC[u_{h_\ell} - u_{h_{\ell-1}}]  \|_{L^2(\Omega;X)}.
  \end{align*}
  Next we apply the
  assumptions~\eqref{MLMC:assump1}--\eqref{MLMC:assump3} and
  Lemma~\ref{lemma:general MC error} to obtain the asserted orders for
  the error by calculating
  \begin{equation}\label{MLMC error1}
    \begin{split}
      \RMSE 
      &\le \| \E[u-u_{h_L}] \|_X  +  M_0^{-1/2} \sigma[ u_{h_0} ]\\
      &\quad + \sum_{\ell=1}^{L} M_\ell^{-1/2} \sigma[ u_{h_\ell} - u_{h_{\ell-1}}  ]\\
      &\le C_1h_L^\alpha + C_{00}  M_0^{-1/2} + C_0 \sum_{\ell=1}^{L} M_\ell^{-1/2} h_{\ell-1}^\beta\\
      &= O(h_L^\alpha)+O(M_0^{-1/2}) + \sum_{\ell=1}^{L}O(M_\ell^{-1/2}) O(h_{\ell-1}^\beta).
    \end{split}
  \end{equation}
  This concludes the proof.
\end{proof}

\section{Optimal Monte-Carlo and Multi-Level Monte-Carlo Methods}
\label{s:complexity}

In this section, we first estimate the computational cost of the MLMC
FE method to achieve a given accuracy and compare it with the MC FE
method. Based on these considerations, the computational work is then
minimized for a given accuracy to be achieved in order to find the
optimal number of samples and the optimal mesh size.

As the model equations \eqref{modeleqn} are a system of PDEs, the work
estimate consists of the sum of the work for all equations, i.e., the
Poisson equation for~$V$ and the two drift-diffusion equations for~$u$
and~$v$. Therefore, the total computational work is given by
\begin{equation}\label{eq:cost}
  W := W_P + 2W_D = W_{P,a} + W_{P,s} + 2W_{D,a} + 2W_{D,s},
\end{equation} 
where the index~$P$ indicates the Poisson equation, the index~$D$
indicates the two drift-diffusion equations, the index~$a$ denotes
assembly of the system matrix, and the index~$s$ denotes solving the
system matrix. We assume that the necessary number of fixed-point or
Newton iterations to achieve numerical convergence is constant; this
is supported by the numerical results.

For each of these four parts the work in level~$\ell$
is given by
\begin{subequations}\label{e:work}
  \begin{align}
    (W_\ell)_{P,a} &= \mu_1 M_\ell h_\ell^{-\gamma_1},\\
    (W_\ell)_{P,s} &= \mu_2 M_\ell h_\ell^{-\gamma_2},\\
    (W_\ell)_{D,a} &= \mu_3 M_\ell h_\ell^{-\gamma_3},\\
    (W_\ell)_{D,s} &= \mu_4 M_\ell h_\ell^{-\gamma_4}
  \end{align}
\end{subequations}
with all $\mu_k > 0$ and $\gamma_k > 0$.  Here $M_\ell$
is the number of samples used at level~$\ell$, and $h_\ell$ is the
corresponding mesh size.

Analogously, in the case of the vanilla Monte-Carlo method, the
computational work is obtained without stratification, i.e., there is
only one level. In this case, we will drop the index~$\ell$.

The exponents (and constants) in equations \eqref{e:work} are
determined by the algorithm used for assembling the FE matrix in the
case of $W_{P,a}$ and $W_{D,a}$ (see, e.g.,
\cite{Cuvelier2013efficient} for an efficient algorithm) and by the
order of the FE discretization in the case of $W_{P,s}$ and $W_{D,s}$
(see Section \ref{s:MC-FEM}). The constants $\mu_i > 0$ depend on the
implementation.

\subsection{The Optimal Monte-Carlo Finite-Element Method}
\label{s:MC work}

In the case of the Monte-Carlo method, there is only one level so that
the index~$\ell$ will be dropped. We will choose the optimal~$M$
and~$h$ such that the total computational cost~$W$ is minimized given
an error bound~$\epsilon$ to be achieved. This optimization problem
with inequality constraints can be solved using the Karush-Kuhn-Tucker
(KKT) conditions, which are generalization of Lagrange multipliers in
the presence of inequality constraints.

In view of \eqref{e:work} and \eqref{MC-error1}, the most general
problem is the following. We minimize the computational work subject
to the accuracy constraint $\RMSE \le \varepsilon$, i.e., we solve the
optimization problem
\begin{equation}
  \begin{aligned}
    & \underset{M,h}{\text{minimize}}
    & &f(M,h):= \sum_{k=1}^N \mu_k M h^{-\gamma_k}\\
    & \text{subject to}
    & & g_1(M,h):=\varepsilon-\frac{C_0}{\sqrt{M}} - C_1 h^{\alpha} \geq 0,\\
    &&& g_2(M):= M  -1\geq0,\\
    &&& g_3(h):= h -\xi\geq0.
    \label{op:MC}
  \end{aligned}
\end{equation}
Here, $N\in\mathbb{N}$ and $\xi$ is the smallest positive normalized
floating-point number representable. Due to the exponents of~$h$
and~$M$, it is a nonlinear constraint optimization problem. Our goal
is to formulate the inequality constrained problem as an equality
constrained problem to which Newton's method can be applied. In order
to solve the optimization problem, we use the interior-point method
\cite{forsgren2002interior, wright2005interior}.

For each $\mu>0$, we replace the non-negativity constraints with
logarithmic barrier terms in the objective function
\begin{equation}
  \begin{aligned}
    & \underset{\chi,s}{\text{minimize}}
    & & f_{\mu}(\chi,s):=  f(\chi)-\mu\sum_{i} \ln(s_i)\\
    & \text{subject to}
    & & g(\chi)-s=0.
    \label{opt:MC1}
  \end{aligned}
\end{equation}
Here $\chi$, a vector, denotes $(M,h)$ and the vectors~$g$ and~$s$
represent the $g_i(x)$ and $s_i$, respectively. The $s_i$ are
restricted to be positive away from zero to ensure that the $\ln(s_i)$
are bounded. As $\mu$ decreases to zero, the minimum of $f_\mu$
approaches the minimum of~$f$. After denoting the Lagrange multiplier
for the system (\ref{opt:MC1}) by~$y$, the system
\begin{align*}
  \nabla f(\chi)-\nabla g(\chi)^T y&=0,\\
  SYe&=\mu e,\\
  g(\chi)-s&=0
\end{align*}
is obtained, where $S$ is a diagonal matrix with elements $s_i$, $e$
is a vector of all ones, and $\nabla g$ denotes the Jacobian of the
constraint~$g$. Now we apply Newton's method to compute the search
directions $\Delta \chi$, $\Delta s$, $\Delta h$ via
\begin{equation}\label{Newton}
  \begin{pmatrix}
    H(\chi,y) & 0 & -A(\chi)^T\\
    0 & Y & S\\
    A(x) &-I & 0
  \end{pmatrix} 
  \begin{pmatrix}
    \Delta \chi\\
    \Delta s  \\
    \Delta h    
  \end{pmatrix}
  =
  \begin{pmatrix}
    -\nabla f(\chi)+A(\chi)^Ty\\
    \mu e-SYe  \\
    -g(\chi)+s    
  \end{pmatrix}.
\end{equation}
The Hessian matrix is given by
\begin{equation*}
  H(\chi,y)=\nabla^2 f(\chi)-\sum_{i}y_i\nabla^2 g_i(\chi)
\end{equation*}
and $A(\chi)$ is the Jacobian matrix of the constraint (\ref{op:MC}).
The second equation is used to calculate $\Delta s$. By substituting
into the third equation, we obtain the reduced KKT system
\begin{equation}\label{Newton1}
  \begin{pmatrix}
    -H(\chi,y) & A(\chi)^T\\
    A(\chi) & SY^{-1} \\
  \end{pmatrix} 
  \begin{pmatrix}
    \Delta \chi\\
    \Delta s  \\  
  \end{pmatrix}
  =
  \begin{pmatrix}
    \nabla f(\chi)-A(\chi)^Ty\\
    -h(\chi)+\mu Y^{-1}e
  \end{pmatrix}.
\end{equation}
Now we use iteration to update the solutions by
\begin{align*}
  \chi^{(k+1)}&:=\chi^{(k)}+\alpha^{(k)}\Delta\chi^{(k)},\\
  s^{(k+1)}&:=s^{(k)}+\alpha^{(k)}\Delta s^{(k)},\\
  y^{(k+1)}&:=y^{(k)}+\alpha^{(k)}\Delta y^{(k)},
\end{align*} 
where $(\chi^{(0)},s^{(0)},y^{(0)})$ is the initial guess and
$\alpha^{(k)}$ is chosen to ensure both that $s^{(k+1)}>0$ and the
objective function
\begin{equation*}
  \Psi_{\upsilon,\mu}(\chi,s)=f_{\mu}(\chi,s)+\frac{\upsilon}{2}\|g(\chi)-s\|,
\end{equation*}  
is sufficiently reduced \cite{benson2000interior}. The parameter
$\upsilon$ may increase with the iteration number to force the
solution toward feasibility.

\subsection{The Optimal Multi-Level Monte-Carlo Finite-Element Method}
\label{s:optimal}

For an optimal multi-level Monte-Carlo finite-element method, our goal
is to determine the optimal hierarchies $(L, \left\lbrace M_{\ell}
\right\rbrace_{\ell=0}^L, h_0,r)$ which minimize the computational
work subject to the given accuracy constraint $\RMSE \le \varepsilon$.
The optimal number~$L$ of levels is also unknown a~priori. To this
end, we solve the optimization problem
\begin{equation}
  \begin{aligned}
    & \underset{M_\ell,h_0,r}{\text{minimize}}
    & & f(M_\ell, h_0,r,L) :=
    \sum_{\ell=0}^{L}\sum_{k=1}^N \mu_k M_\ell h_0^{-\gamma_k} r^{\ell\gamma_k}\\
    & \text{subject to}
    & & g_1(M_\ell,h_0,r,L):= \varepsilon- \frac{C_{00}}{\sqrt{M_0}}-C_0 \sum_{\ell=1}^{L}\frac{1}{\sqrt{M}_\ell} h_0^{\beta} r^{-(\ell-1)\beta} - C_1 h_0^{\alpha} r^{-L\alpha} \geq 0,\\
    &&&  g_2(M_\ell):= M_\ell  -1\geq0,\qquad \ell = 0, \ldots, L,\\
    &&&  g_3(h_0):= h_0 -\xi\geq0,\\
    &&& g_4(r):= r -1\geq0,
    \label{op:MLMC}
  \end{aligned}
\end{equation}
where $N\in\mathbb{N}$ and $\xi$ is the smallest positive normalized
floating-point number representable by the machine.

Similar to the vanilla Monte-Carlo case, we use the interior-point
method to solve this nonlinear problem and optimize the hierarchies.
In problems with two or three physical/spatial dimensions, the optimal
determination of the mesh sizes $h_\ell$ is a crucial factor in the
optimization problem specifically if the exponents $\gamma_k$ are
greater than~$1$.

There are two options: one is to choose the $h_\ell$ as a geometric
progression according to \eqref{h_ell}. In this case, we solve the
minimization problem \eqref{op:MLMC}. The other is to choose the mesh
sizes $h_\ell$ freely such that they only satisfy the natural
condition
\begin{equation*}
  h_0\geq h_1\geq h_2\geq \cdots \geq h_L.
\end{equation*} 
We will explore both options in Section \ref{s:implementation}.

In the second case, when the mesh sizes are freely chosen, we write
them as
\begin{equation*}
  h_\ell:=\frac{h_0}{r_\ell},\qquad  \ell=1,\ldots,L,
\end{equation*}
where $r_\ell:=\prod_{i=1}^{\ell}r_i$ with $r_i\ge 1$. It is clear
that $r_L\geq r_{L-1} \ldots\geq r_1\geq1$. Hence, in \eqref{op:MLMC},
we may replace~$r$ by $r_\ell$ and then the optimization problem can
be rewritten as
\begin{equation}
  \begin{aligned}
    & \underset{M_\ell,h_0,r_\ell}{\text{minimize}}
    & & f(M_\ell, h_0,r_\ell,L) :=
    \sum_{\ell=0}^{L}\sum_{k=1}^N \mu_k M_\ell h_0^{-\gamma_k} r_\ell^{\gamma_k}\\
    & \text{subject to}
    & & g_1(M_\ell,h_0,r_\ell,L):= \varepsilon- \frac{C_{00}}{\sqrt{M_0}}-C_0 \sum_{\ell=1}^{L}\frac{1}{\sqrt{M}_\ell} h_0^{\beta} r_{\ell-1}^{-\beta} - C_1 h_0^{\alpha} r_L^{-\alpha} \geq 0,\\
    &&&  g_2(M_\ell):= M_\ell  -1\geq0,\qquad \ell = 0, \ldots, L,\\
    &&&  g_3(h_0):= h_0 -\xi\geq0,\\
    &&& g_4:= r_i -1\geq0,\quad 1\le i\le \ell \quad \text{and} \quad \ell=1,\ldots,L.
    \label{op:MLMC2}
  \end{aligned}
\end{equation}

In the next section, we apply these two approaches to a MLMC FE method
and discuss their efficiency.
 
\section{Numerical Results}\label{s:implementation}

In this section, we present numerical results for the Monte-Carlo and
multi-level Monte-Carlo methods for the drift-diffusion-Poisson
system. We also investigate the choices of the FE mesh sizes on each
level, namely as geometric progressions or freely chosen. The random
coefficients in the drift-diffusion-Poisson system considered here
stem from a real-world application, namely the effect of random
dopants in nanoscale semiconductor devices, which is of great
importance in its own right.

\subsection{The Leading Example}

Random-dopant effects are called discrete-dopant fluctuation effects
\cite{van2004principles, roy2005dopants, khodadadian2015basis}. In
nanoscale semiconductor devices, the charge profile of the dopant
atoms cannot be validly modeled as a continuum anymore, but the random
location of each dopant needs to be taken into account. This means
that each device is a realization of a random process and corresponds
to an event~$\omega$. In this manner, the potential and
carrier-density fluctuations due to the discreteness and randomness of
the dopants are clearly captured.


Here the silicon lattice is doped with boron as the impurity atoms.
The domain $D\subset\mathbb{R}^2$ is depicted in
Figure~\ref{fig:mesh}. The thickness of the oxide layer is
$\unit{8}{nm}$, the thickness of the nanowire is $\unit{50}{nm}$, and
its width is $\unit{60}{nm}$. Regarding the geometry, Dirichlet
boundary conditions are used at the contacts with a back-gate voltage
of $\unit{-1}{V}$ (at the bottom of the device) and an electrode
voltage of $\unit{0}{V}$ (at the top of the device). Zero Neumann
boundary conditions are used everywhere else.
The relative permittivities in the
subdomains are $A_\mathrm{Si} = 11.7$, $A_\mathrm{ox}=3.9$,
$A_\mathrm{liq}=78$, and $A_\mathrm{dop}=4.2$. The number of dopants
placed randomly in the device corresponds to a  doping
concentration of  $\unit{10^{16}}{cm^{-3}}$.
According to its volume, the silicon subdomain hence contains $6$
negative impurity atoms when $\Cdop = \unit{5\cdot10^{15}}{cm^{-3}}$
and $600$ dopants when $\Cdop = \unit{5\cdot10^{17}}{cm^{-3}}$.

In order to solve the system of equations, we use Scharfetter-Gummel
iteration.
In spite of the quadratic convergence of Newton's method for the
system, Scharfetter-Gummel iteration has advantages for the problem at
hand. First of all, Scharfetter-Gummel iteration is much less
sensitive to the choice of the initial guess than Newton's method.
Another important feature is the reduced computational effort and
memory requirement, since in each iteration it requires the successive
solution of three much smaller elliptic problems.

\begin{figure}[ht!]
  \centering
  \subfloat{\includegraphics[width=0.5\linewidth]{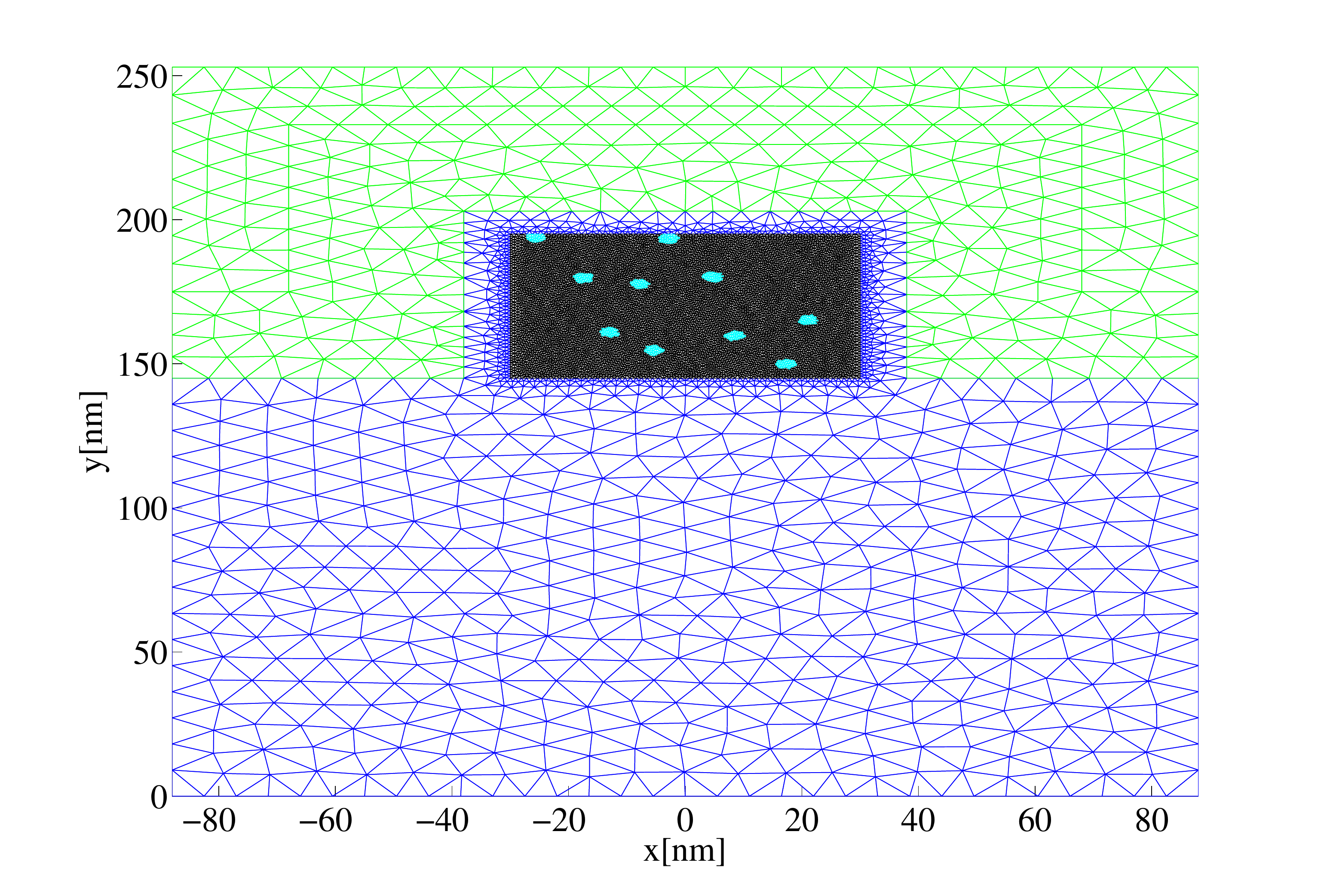}}
  \hfill    
  \subfloat{\includegraphics[width=0.5\linewidth]{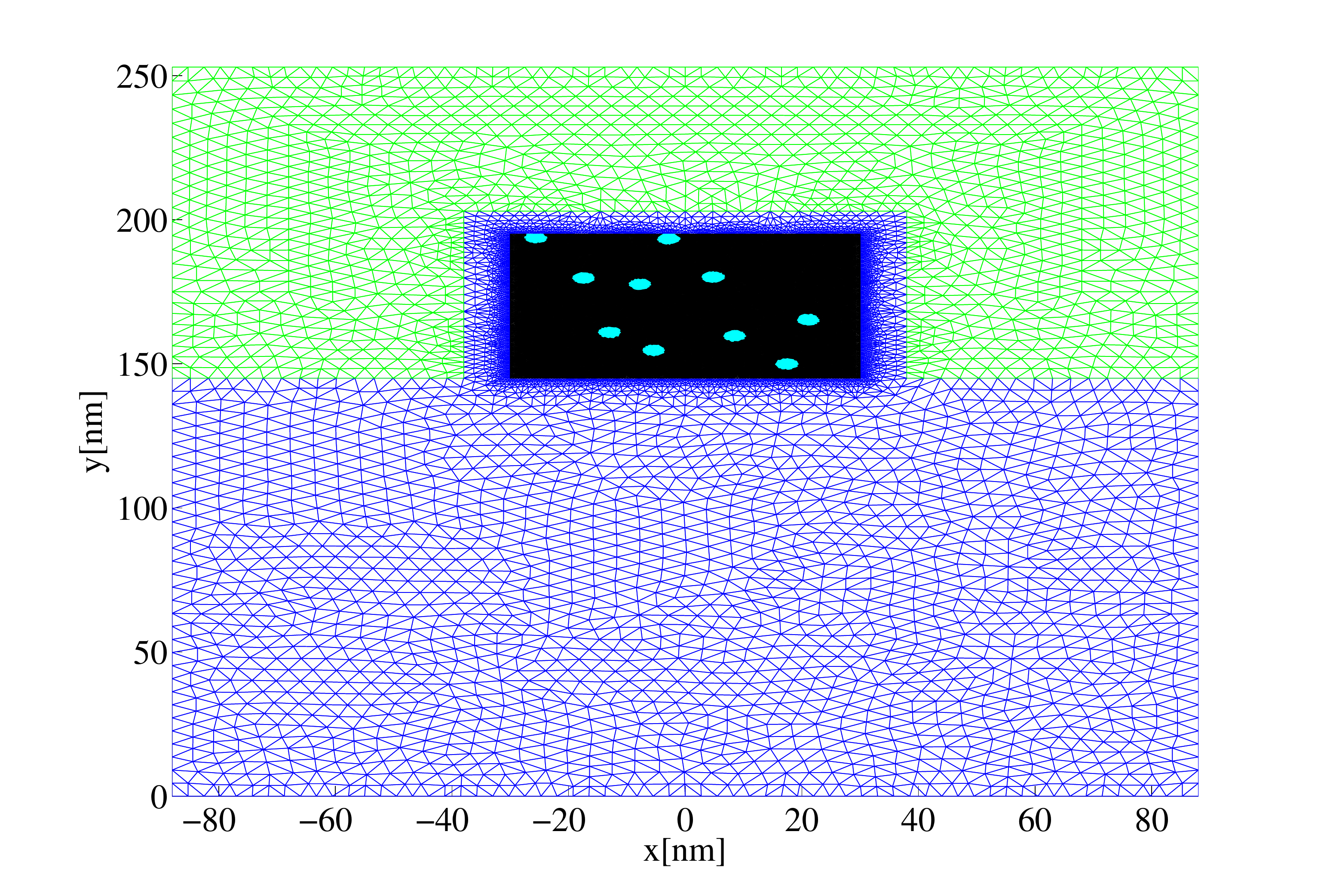}}
  \caption{Meshes for the random distribution of impurity atoms (cyan
    circles) in a nanowire field-effect sensor for levels $\ell=0$
    (left) and $\ell=1$ (right), where $h_0=5$, $r=2$, and $\Cdop =
    \unit{2\cdot10^{17}}{cm^{-3}}$. Additionally, oxide subdomain ($D_{\text{ox}}$), trandsucer ($D_{\text{Si}}$) and the electrolyte ($D_{\text{liq}}$) are depicted with blue, balck and green meshes respectively. }
  \label{fig:mesh}
\end{figure}
  
\begin{figure}[ht!]
  \centering
  \subfloat{\includegraphics[width=0.5\linewidth]{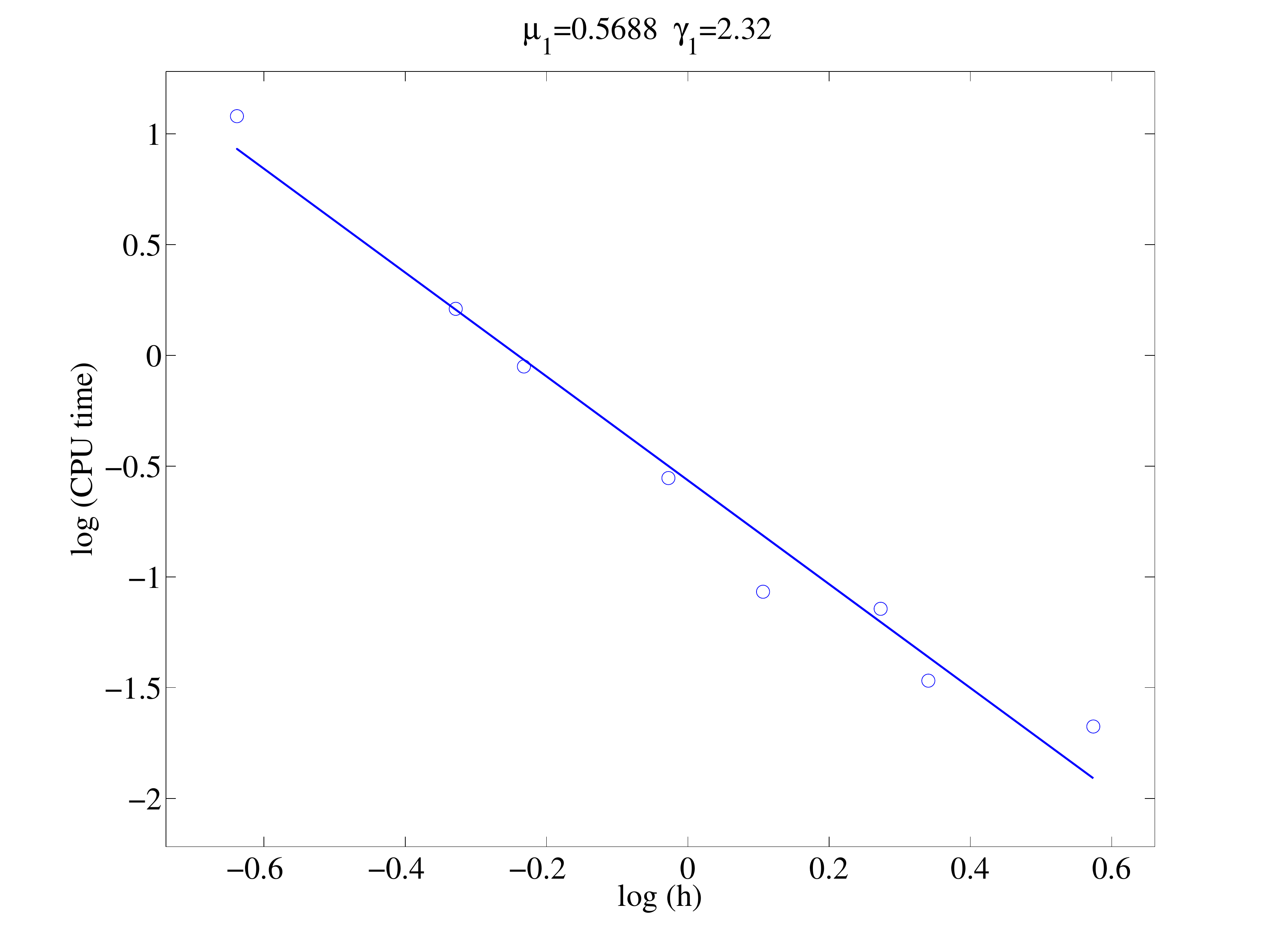}}
  \hfill
  \subfloat{\includegraphics[width=0.5\linewidth]{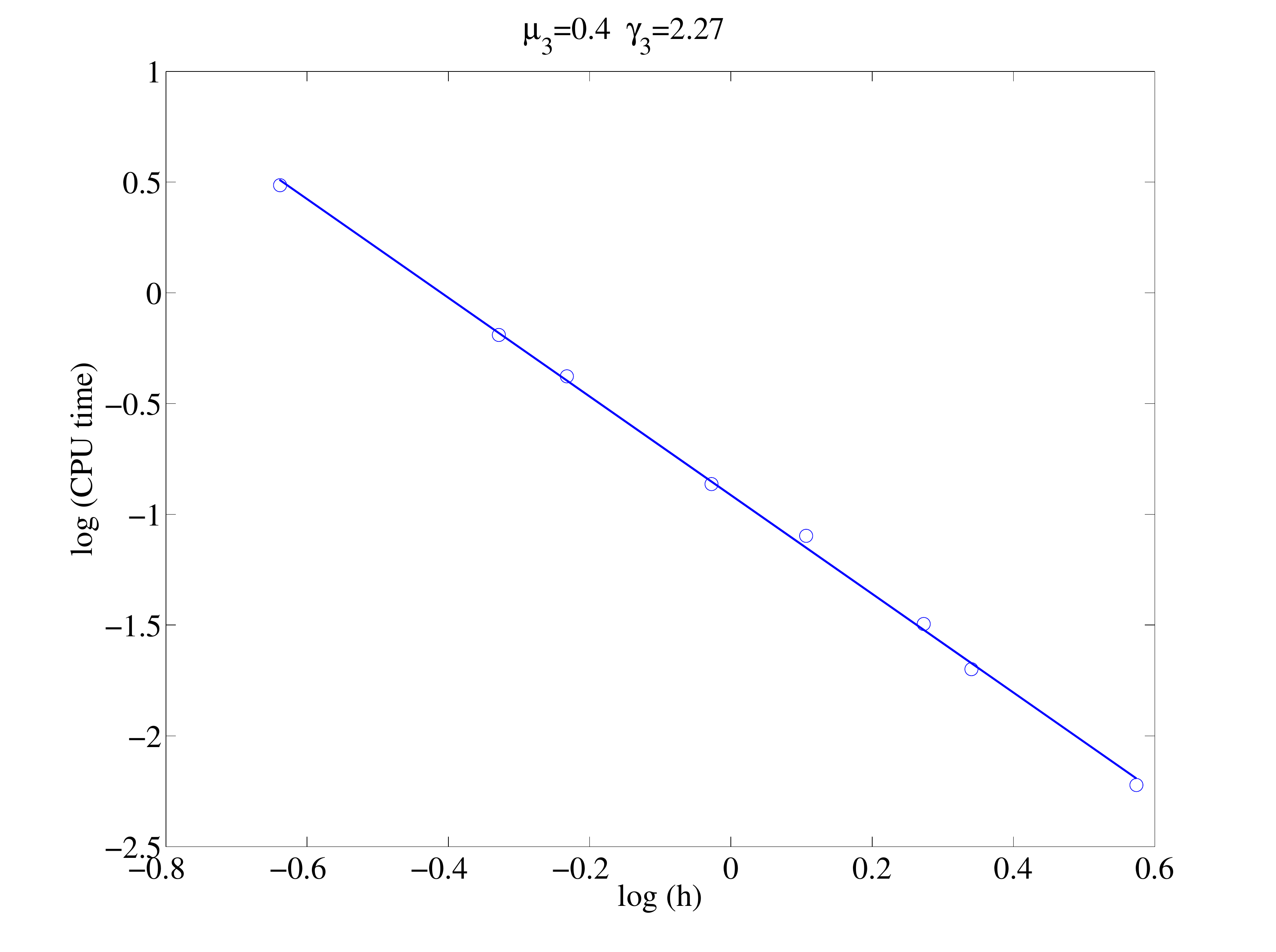}}
  \newline 
  \subfloat{\includegraphics[width=0.5\linewidth]{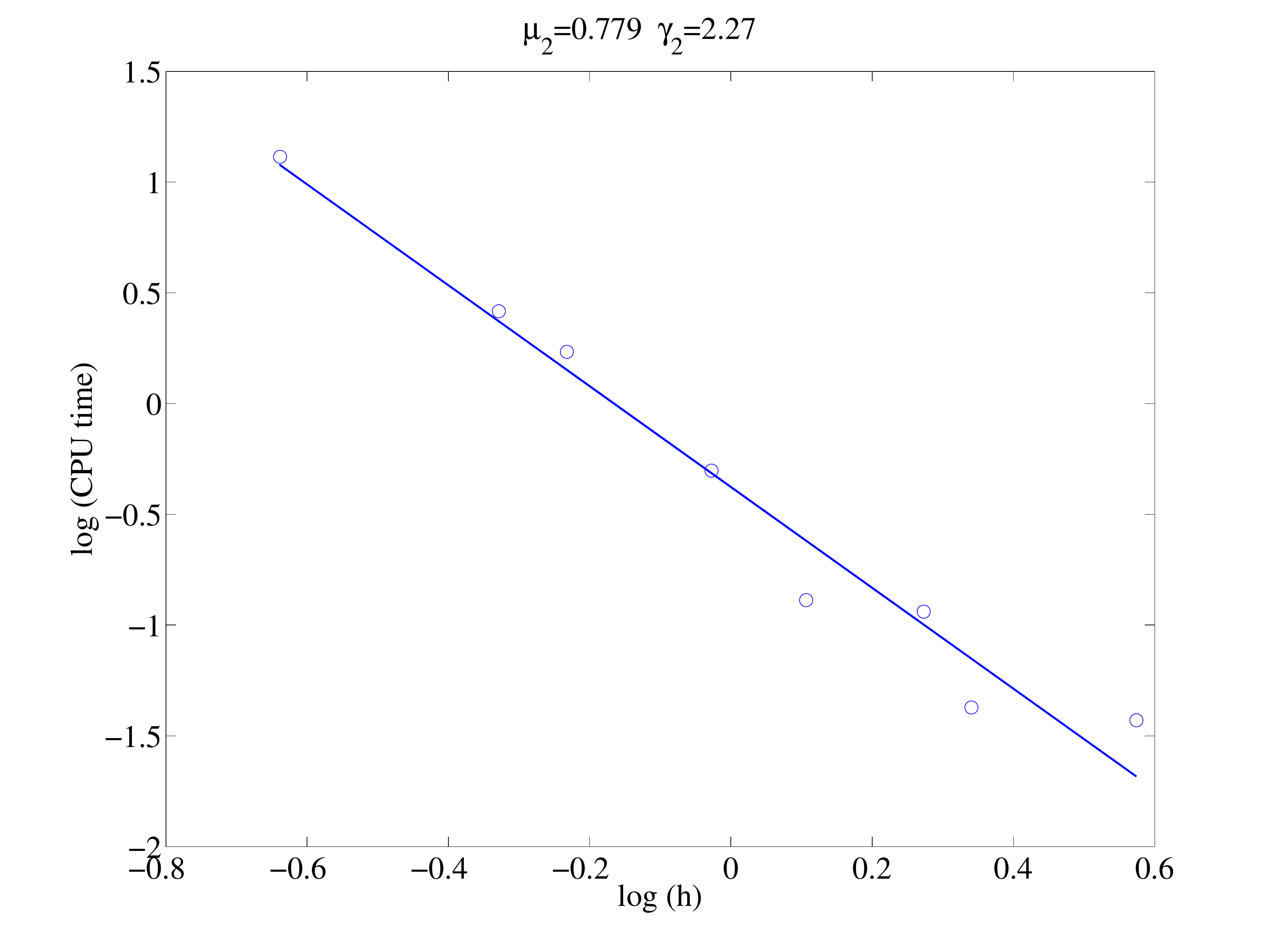}}
  \hfill
  \subfloat{\includegraphics[width=0.5\linewidth]{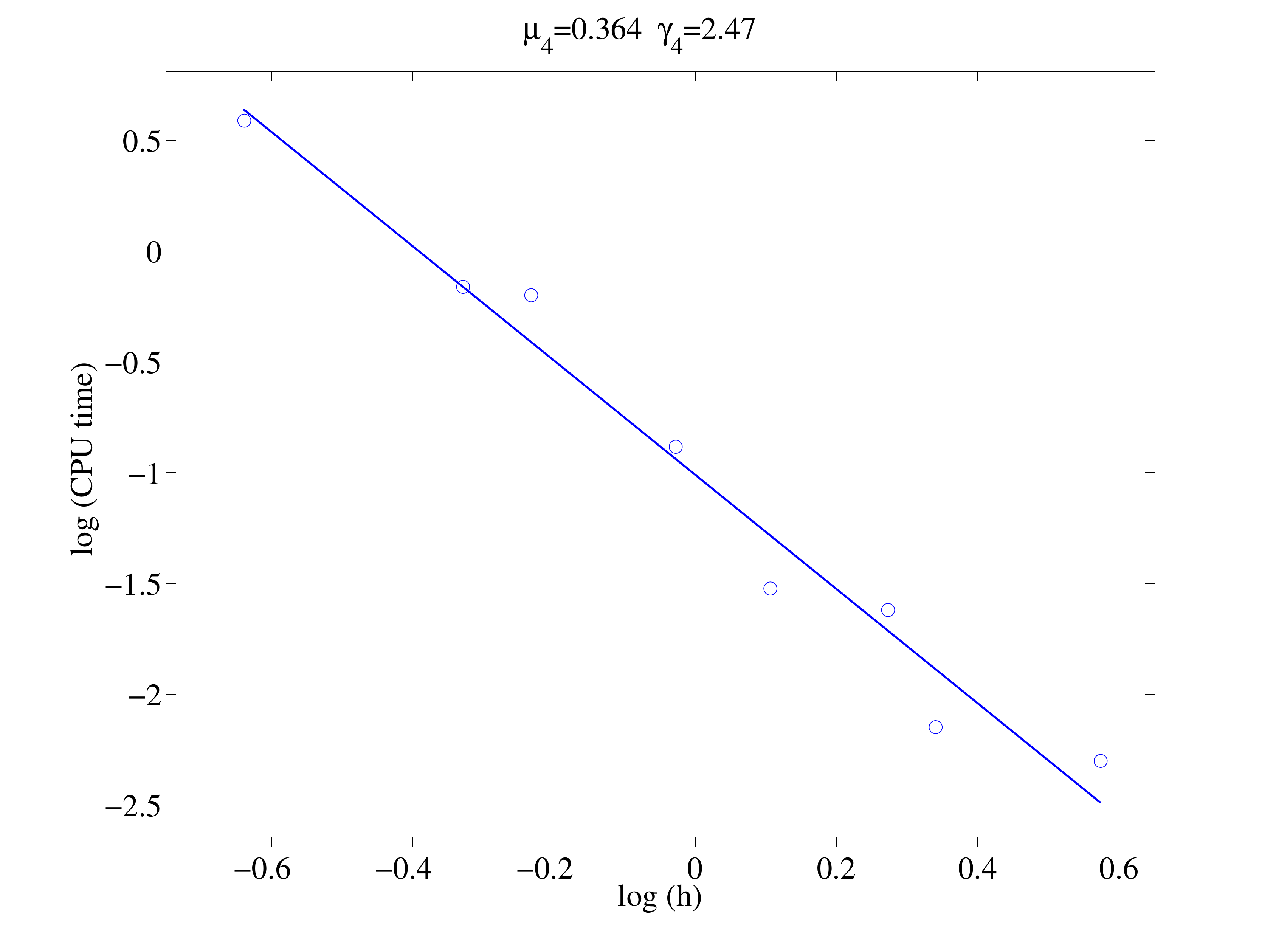}}
  \caption{Computational work for matrix assembly (top) and solving
    the system (bottom), both for the Poisson equation (left) and the
    drift-diffusion-equations (right).}
  \label{fig:work_MC}
\end{figure} 


\subsection{The Computational Work}

As the first step, we calculate the coefficients in the expressions
\eqref{e:work} for the computational work. To that end, we  solve the system for various mesh sizes and
measure the time spent on matrix assembly and solving the resulting
system, both for the Poisson equation and the drift-diffusion
equations. Figure~\ref{fig:work_MC} shows the results for the
coefficients in the expressions for the computational work.


The coefficients~$\alpha$ and~$C_1$ in the FE discretization error
\begin{equation*}
  \|\mathbb{E}[V-\hat{V}_h  ]\|_X  +  \|\mathbb{E}[u-\hat{u}_h  ]\|_X
  + \|\mathbb{E}[v-\hat{v}_h  ]\|_X\leq C_1h^\alpha
\end{equation*}
of the system are given in Figure~\ref{fig:MLMC:aplha}.  The exponent
$\alpha=0.96$ found here agrees very well with the order of the
discretization used here, i.e., $P_1$ finite elements.

For the statistical error, we determine the coefficients in the
inequality
\begin{equation*}
  (\sigma[\Delta V_{h_0}]+ \sigma[\Delta V_{h_\ell}])
  + (\sigma[\Delta u_{h_0}]+ \sigma[\Delta u_{h_\ell}])
  + (\sigma[\Delta v_{h_0}]+ \sigma[\Delta v_{h_\ell}])
  \le C_{00}+C_0 h_{\ell-1}^{\beta}.
\end{equation*}
Here $C_{00}=0.197$ and the rest of the coefficients are shown in
Figure~\ref{fig:MLMC:aplha}.

\begin{figure}[ht!]
  \centering
  \subfloat{\includegraphics[width=0.5\linewidth]{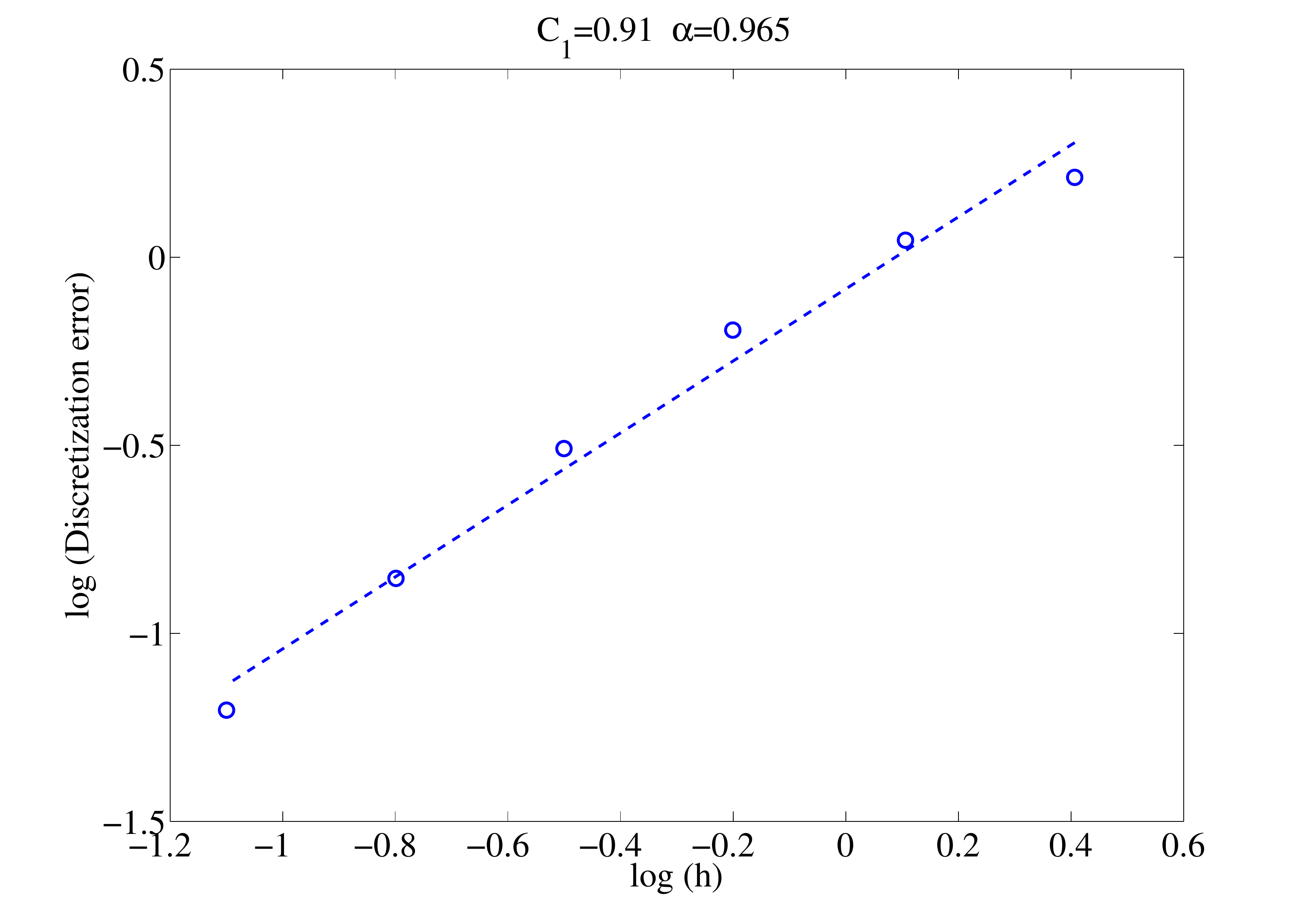}}
  \hfill    
  \subfloat{\includegraphics[width=0.5\linewidth]{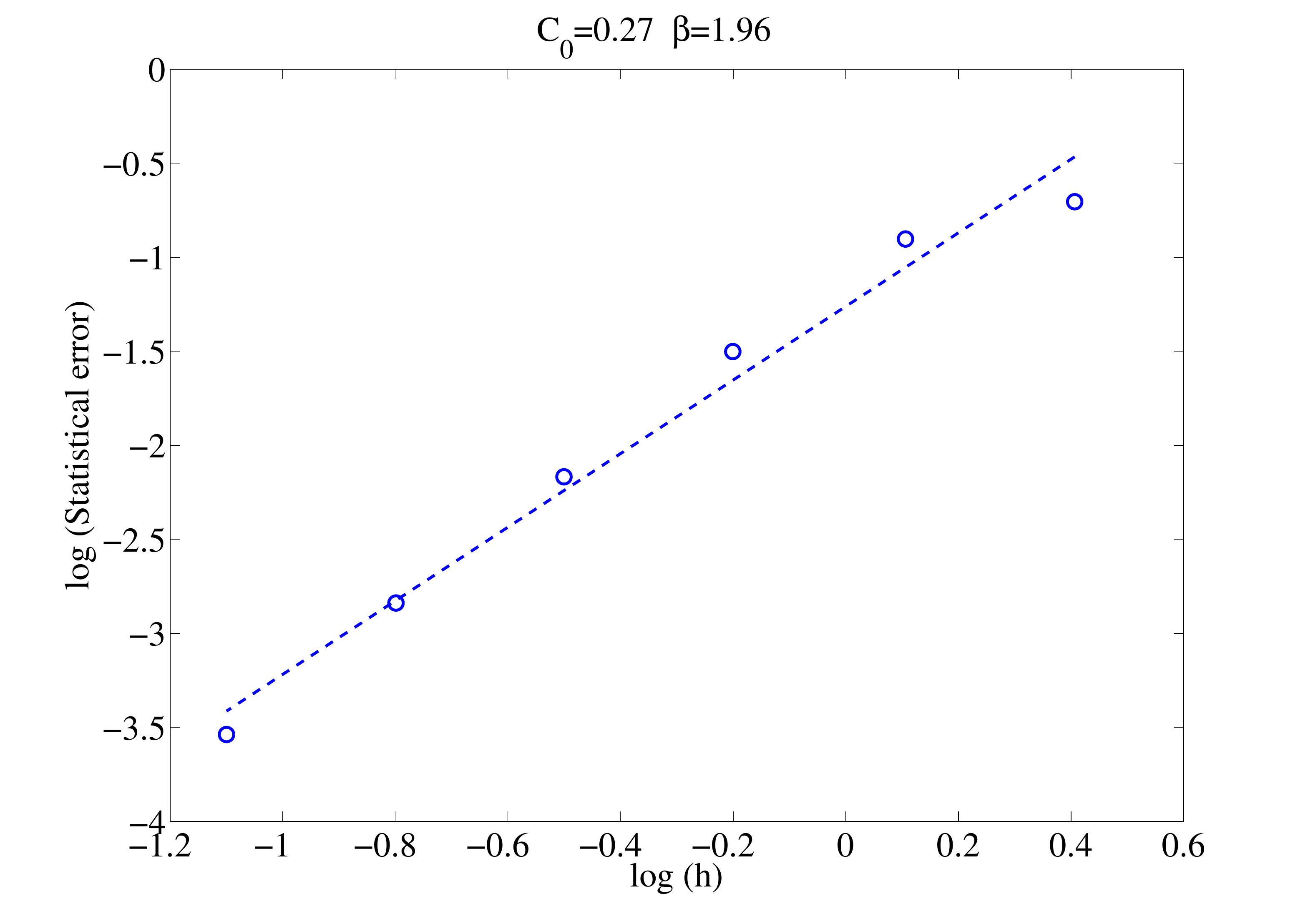}}
  \caption{Discretization error (left) and statistical (right) error
    as a function of~$h$.}
  \label{fig:MLMC:aplha}
\end{figure}

\subsection{Optimization}

Having determined the coefficients in the expressions for the
computational work, it is now possible to numerically solve the
optimization problems. As described in Section~\ref{s:complexity}, we
apply an iterative interior-point method to optimize both the number
of samples and mesh sizes. The results here are obtained for
$\xi=2^{-52}$.

\subsubsection{Monte Carlo}

First of all, we solve the optimization problem \eqref{op:MC} for the
MC-FE method. Because there is only one level, it is straightforward
to solve. The optimal values for the MC FE method are summarized in
Table~\ref{table:3} for given $\varepsilon$.

\begin{table}[ht!]
  \centering
  \begin{tabular}{rrrrrrrr}
    $\varepsilon$ & 0.1  & 0.05 & 0.03 & 0.02 & 0.01 & 0.005 &0.002\\
    \hline
    $h$ & 0.054 & 0.026  & 0.015 & 0.010& 0.005 & 0.002 &  0.001\\
    $M$ & 19 & 77   & 214 & 483 & 1940 & 7785 & 48\,845\\
  \end{tabular}
  \caption{Optimal MC FE method parameters for various given error
    tolerances.}
  \label{table:3}
\end{table} 

\subsubsection{Multi-Level Monte Carlo}

In the MLMC-FE method, determining the optimal number of levels is an
important part of the calculation. This is achieved here by solving
the optimization problem for several levels starting with a single
level and noting that the computational work increases above a certain
number of levels. More precisely, we solve the optimization problem
\eqref{op:MLMC} for $0\leq L\leq 8$ levels as well as for various
given error bounds.

Since the number of samples in each level is a continuous variable in
the optimization problem, the optimal number of levels is -- in
general -- not an integer and hence we choose $\floor{M_{\ell}}$,
$\ell = 0, \ldots, L$, as the final numbers of levels. Therefore, due
to the second constraint $g_2$, $M_\ell$ is an integer greater equal~$1$
for all levels.

\begin{figure}[ht!]
  \centering
  \subfloat{\includegraphics[width=0.5\linewidth]{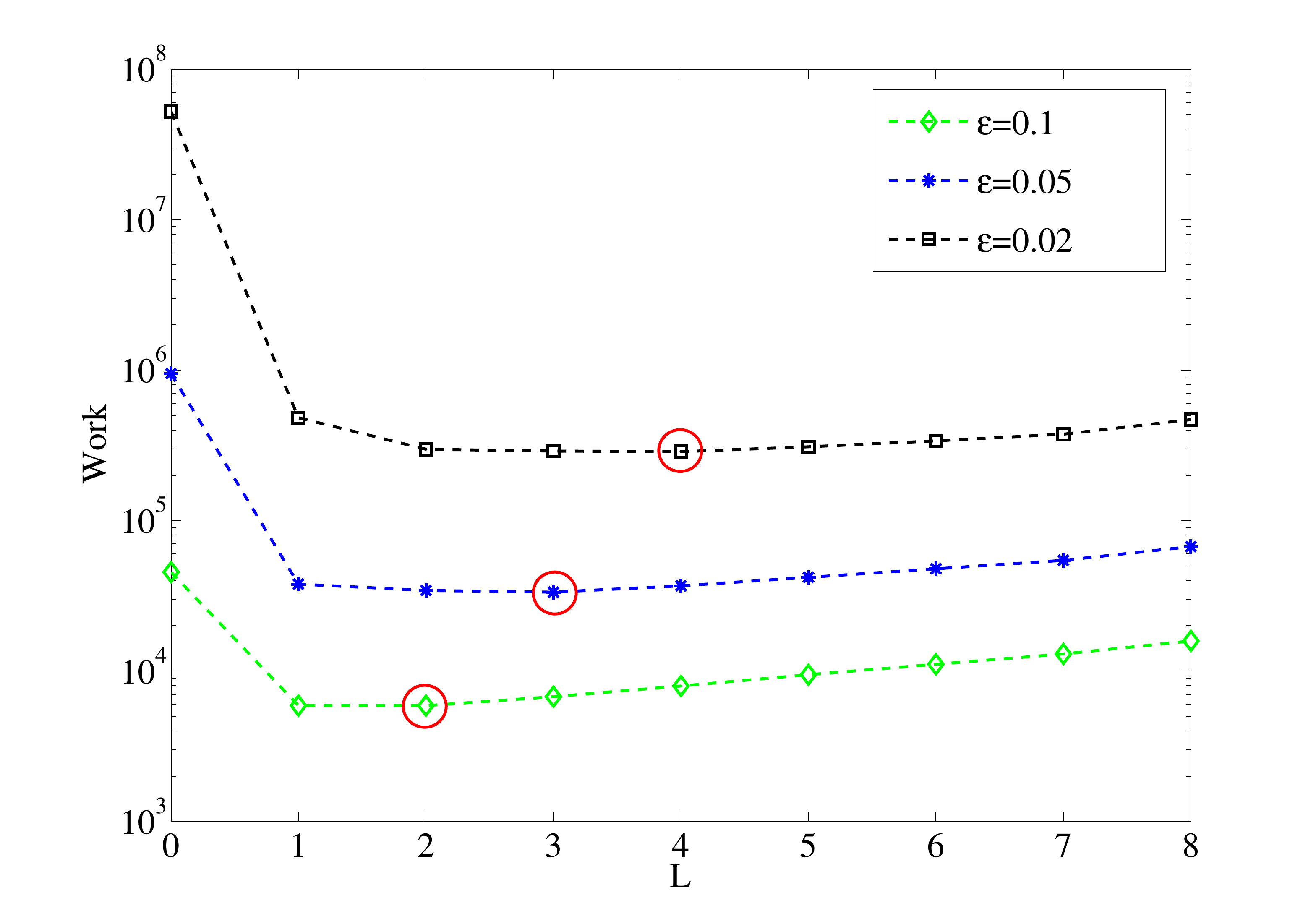}}
  \hfill    
  \subfloat{\includegraphics[width=0.5\linewidth]{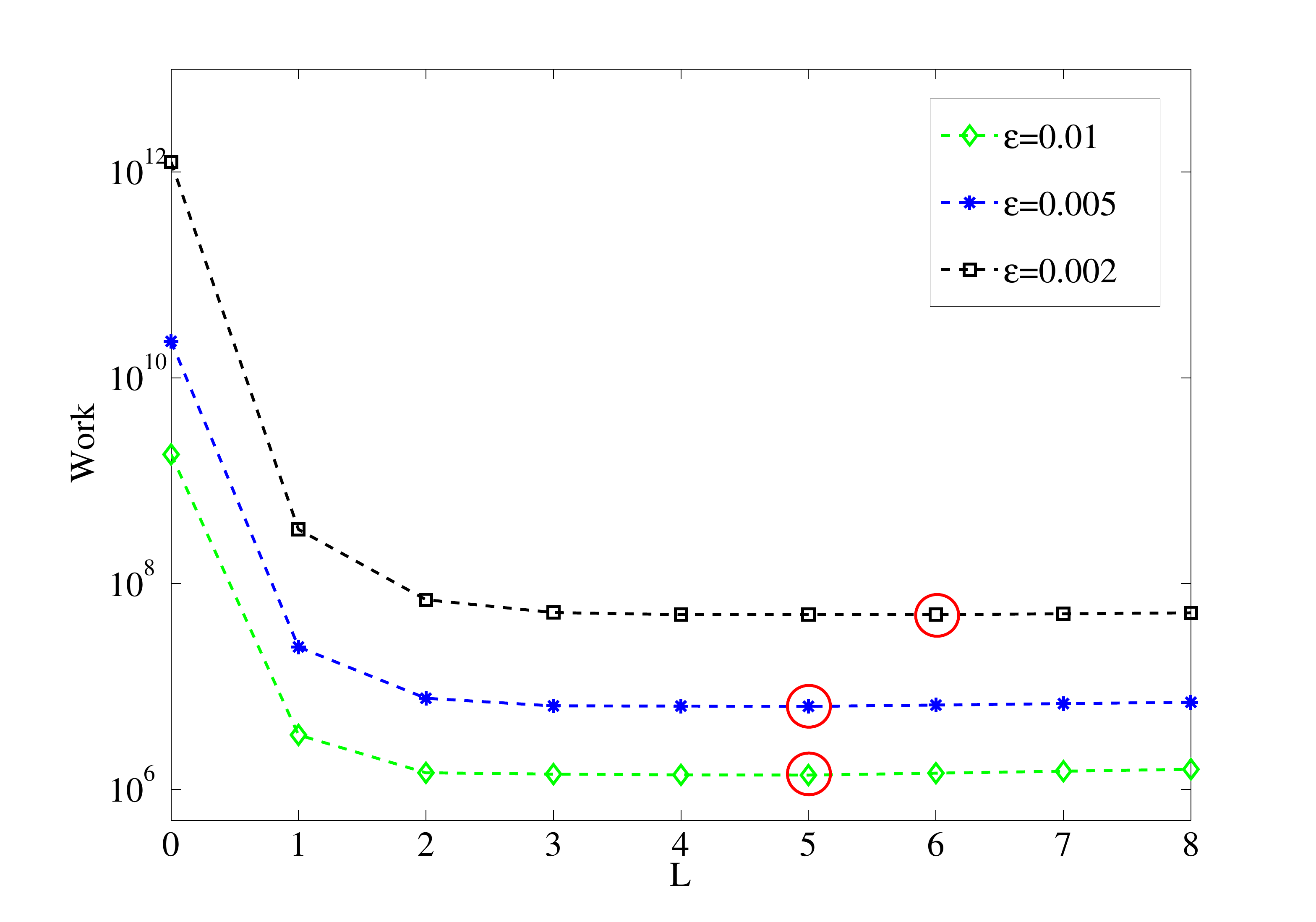}}
  \caption{The minimized computational work for the MLMC-FE method as
    a function of the number of levels and as a function of the given
    error tolerance. The results for a geometric progression for~$h$
    (left) and general~$h$ (right) are shown. The number of levels
    yielding the minimal overall computational work is indicated by
    red circles.}
  \label{fig:MLM_comp1}
\end{figure}

\begin{figure}[ht!]
  \centering
  \includegraphics[width=0.6\linewidth]{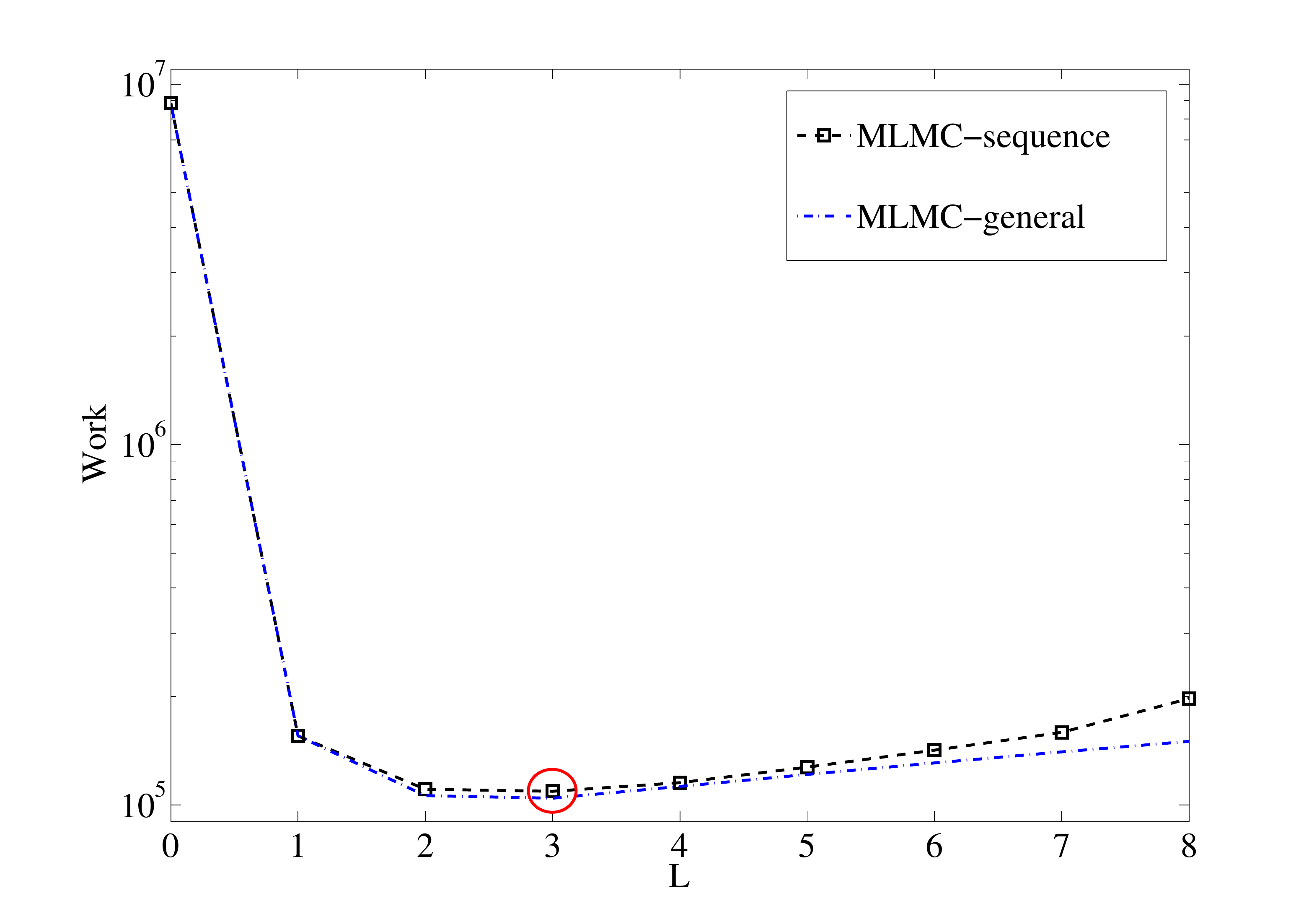}
  \caption{Comparison between the two different approaches to MLMC FE method
    for $\varepsilon=0.03$.}
  \label{fig:MLMC_comp3}
\end{figure}

The results of the optimization problems provide insight into the MLMC
procedure. Figure~\ref{fig:MLM_comp1} shows the minimized
computational work as a function of the number of levels and as a
function of the given tolerance. It shows that for smaller tolerances,
more levels are necessary to obtain the smallest computational work.

In Figure~\ref{fig:MLMC_comp3}, the two approaches to multi-level
Monte Carlo are compared, namely choosing the $h_\ell$ as a geometric
progressions or freely. Due to generality of the second option, the
total work when choosing the $h_\ell$ freely is lower compared to the
first option. The results for both approaches to MLMC-FEM are
summarized in Tables~\ref{table:1} and~\ref{table:2} for various given
error tolerances.

\begin{table}[ht!]
  \centering
  \begin{tabular}{*{3}l*{7}r}
    $\varepsilon$& $h_0$ & $r$  & $M_0$ & $M_1$& $M_2$ & $M_3$ & $M_4$&$M_5$ &$M_6$\\
    \hline
    $0.1$ & 0.359 & 2.650  & 59 & 4 & 1 & -- & --&--&--  \\
    $0.05$   & 0.350& 2.430   & 276 & 22 & 2 & 1 & --& --& --\\
    $0.03$     & 0.329 & 2.830 & 779 & 46 & 2 & 1 & --& --&-- \\
    $0.02$     & 0.339 & 2.440 & 1\,993 & 152 & 12 & 1 & 1& --&-- \\
    $0.01$     & 0.347 & 2.355 & 9\,428 & 786 & 69 & 6 & 1& 1&-- \\
    $0.005$     & 0.332 & 2.670 & 39\,142 & 2\,526 & 154 & 9 & 1& 1&-- \\
    $0.002$     & 0.334 & 2.647 & 286\,181 & 18\,986 & 1200 & 76 & 6& 1&1\\
  \end{tabular}
  \caption{Optimal levels for the MLMC-FE method with~$h_\ell$ chosen
    as a geometric progression for given error
    tolerances~$\varepsilon$.}
  \label{table:1}
\end{table} 

\begin{table}[ht!]
  \centering
  \begin{tabular}{*{9}l}
    $\varepsilon$ & $h_0$ & $r_1$ & $r_2$ & $r_3$ & $r_4$& $r_5$& $r_6$\\
    \hline
    $0.1$    & 0.366 & 2.100 & 3.490 &-- & --&-- &--\\
    $0.05$   & 0.360 & 2.168 & 1.860 & 3.712& --&--& --\\
    $0.03$   & 0.343 & 2.351 & 2.138 & 4.750&-- &-- &--\\
    $0.02$   & 0.346 & 2.317 & 2.086 & 1.744& 4.348 &--&--   \\
    $0.01$   & 0.339 & 2.400 & 2.218 & 1.938& 1.536 &4.554&--   \\
    $0.005$  & 0.332 & 2.483 & 2.348 & 2.134& 1.812 &6.220&--   \\
    $0.002$  & 0.328 & 2.541 & 2.442 & 2.280& 2.030 &1.670&7.142  \\
  \end{tabular}
  \bigbreak
  \begin{tabular}{l*{7}r}
    $\varepsilon$  &$M_0 $ & $M_1$ & $M_2$& $M_3$ & $M_4$ & $M_5$& $M_6$\\
    \hline
    $0.1$   & 59 & 6 & 1 & -- & --&--&--  \\
    $0.05$  & 280 & 28 & 4 & 1 & --& --&-- \\
    $0.03$    & 800 &66 & 7 & 1 & --& --&-- \\
    $0.02$     & 2\,012 &171 & 19 & 3 & 1&--&--  \\
    $0.01$    & 8\,897 &697 & 66 & 8 & 2&1&--  \\
    $0.005$    & 38\,251 &2\,778 & 229 & 24 & 4&1&--  \\
    $0.002$     & 271\,031 &18\,667 & 1408 & 124 & 14&3&1  \\
  \end{tabular}
  \caption{Optimal levels for the MLMC-FE method with general~$h_\ell$ for
    given error tolerances~$\varepsilon$.}
  \label{table:2}
\end{table} 
      
\subsubsection{Comparison}

Finally, as Figure~\ref{fig:MLMC_comp2} shows, the computational work
for the multi-level Monte-Carlo method is approximately 10 times lower
than the one for the Monte-Carlo method for larger tolerance levels
such as $\varepsilon=0.1$. The effectiveness of the MLMC-FE method is
more pronounced for smaller error bounds; for $\varepsilon=0.005$, the
computational work is about a factor $10^4$ lower than the Monte-Carlo
work.

The optimal distribution of the samples among the
levels in the multi-level method leads to more evaluations in the
first levels (which are cheaper) and to fewer evaluations in the
higher levels. On the other hand, to satisfy the first constraint of
\eqref{op:MC}, the Monte-Carlo method needs a smaller mesh size
compared to the multi-level method, which greatly increases the total
computational work although the total number of samples is lower.

\begin{figure}[ht!]
  \centering
  \includegraphics[width=0.7\linewidth]{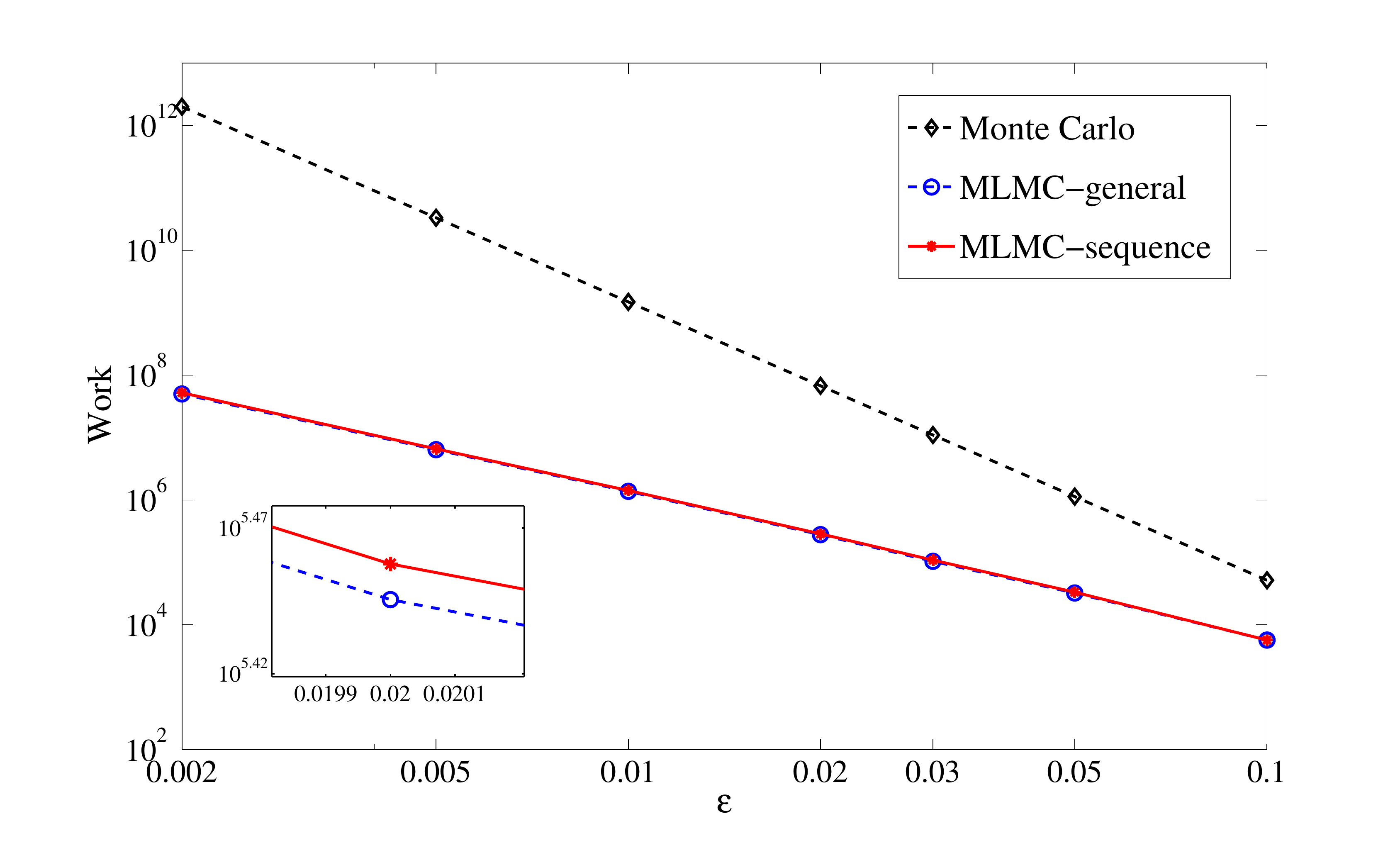}
  \caption{Comparison of total computational work for MC-FEM and the
    two approaches to MLMC-FEM for various given tolerances.}
  \label{fig:MLMC_comp2}
\end{figure}
   
\section{Conclusions}\label{conclusions}

In this work, we considered the stochastic drift-diffusion-Poisson
equations as the main model equation for describing transport in
random environments with many applications. We presented existence and
local uniqueness theorems for the weak solution of the system. We also
developed MC- and MLMC-FE methods for this system of stochastic PDEs.
      
Additionally, we balanced the various parameters in the numerical
methods by viewing this problem as a global optimization problem. The
goal is to determine the numerical parameters such that the
computational work to achieve a total error, i.e., discretization
error plus statistical error, less than or equal to a given error
tolerance is minimized.

Although the exponential terms in the constraints make the
optimization problems nonlinear, the optimization problems can be
solved by an interior-point method with sufficient iterations. The
solution of the constrained optimization problem leads to optimal
($M$, $h$) in the case of the vanilla MC method and to hierarchies
consisting of $(L, \left\lbrace M_{\ell} \right\rbrace_{\ell=0}^L,
h_0,r)$ in the case of the MLMC method.

Moreover, we investigated two different options to the mesh refinement
in the multi-level method. In the comparison of the MC with the MLMC
method, the MLMC method was found to decrease the total computational
effort by four orders of magnitude for small error tolerances. The
speed-up becomes better as the error tolerance decreases.
   
\section{Acknowledgments}
     
The authors acknowledge support by FWF (Austrian Science Fund) START
project no.\ Y660 \textit{PDE Models for Nanotechnology}. The authors
acknowledge discussions with Dr.\ Masoud Ahookhosh (University of
Luxembourg) about numerical optimization methods.

 \bibliographystyle{elsarticle-num}
 \bibliography{MLMC}



\end{document}